\documentclass[11pt]{article}
\usepackage{amsthm,amssymb,amsmath,latexsym}
\usepackage{stmaryrd}		% \llbracket and \rrbracket
\usepackage{hyperref,url}

\usepackage{graphicx}

\usepackage[boxed,noend]{algorithm2e}

\graphicspath{{figures/}}

\newtheorem{theorem}{Theorem}[section]
\newtheorem{lemma}[theorem]{Lemma}
\newtheorem{proposition}[theorem]{Proposition}
\newtheorem{conjecture}[theorem]{Conjecture}

\theoremstyle{definition}
\newtheorem{definition}[theorem]{Definition}
\newtheorem{remark}[theorem]{Remark}
\newtheorem{example}[theorem]{Example}
\newtheorem{question}[theorem]{Open question}

\newcommand{\Q}{\mathbb{Q}}
\newcommand{\F}{\mathbb F}
\newcommand{\Z}{\mathbb Z}

\newcommand{\lam}{\Lambda}
\newcommand{\wt}{\widetilde}
\newcommand{\ol}{\overline}

\newcommand{\seq}[1]{\cite[\href{http://oeis.org/#1}{#1}]{OEIS}}
\newcommand{\nequiv}{\mathrel{\not\equiv}}
\newcommand{\colonequal}{\mathrel{\mathop:}=}

\newcommand{\vcentergraphics}[1]{\ensuremath{\vcenter{\hbox{\includegraphics{#1}}}}}

\begin{document}

\title{Automatic congruences for diagonals of rational functions}

\date{\today}

\author{{\bf Eric Rowland}\\
Universit\'e du Qu\'ebec \`a Montr\'eal, Montr\'eal, Canada\\
%\curraddr{
%	Universit\'e de Li\`ege \\
%	D\'epartement de Math\'ematiques \\
%	Grande Traverse 12 (B37) \\
%	4000 Li\`ege, Belgique
%}
\texttt{rowland@lacim.ca}\\
\\
{\bf Reem Yassawi}\thanks{Partially supported by an NSERC grant.}\\
Trent University, Peterborough, Canada\\
\texttt{ryassawi@trentu.ca}}

\maketitle

\begin{abstract} 
In this paper we use the framework of automatic sequences to study combinatorial sequences modulo prime powers.
Given a sequence whose generating function is the diagonal of a rational power series, we provide a method, based on work of Denef and Lipshitz, for computing a finite automaton for the sequence modulo $p^\alpha$, for all but finitely many primes $p$.
This method gives completely automatic proofs of known results, establishes a number of new theorems for well-known sequences, and allows us to resolve some conjectures regarding the Ap\'ery numbers.
We also give a second method, which applies to an algebraic sequence modulo $p^\alpha$ for all primes $p$, but is significantly slower.
Finally, we show that a broad range of multidimensional sequences possess Lucas products modulo $p$.
\end{abstract}

\section{Introduction}\label{Introduction}

\subsection{Overview}\label{Overview}

A sequence $(a_n)_{n \geq 0}$ of entries in a field $F$ is \emph{algebraic} if its generating function $\sum_{n \geq 0} a_n x^n$ is algebraic over $F(x)$, the field of rational expressions with coefficients in $F$.
A great many combinatorial sequences are algebraic. Examples include the  Catalan and Motzkin numbers, whose generating functions are algebraic over $\Q(x)$, and the Fibonacci sequence, which  satisfies  a linear recurrence with constant coefficients and hence has a rational generating function.

In the past decade, many researchers have been interested in congruences for various algebraic sequences modulo prime powers. 
Deutsch and Sagan~\cite{Deutsch--Sagan} studied arithmetic properties of several sequences, including the Catalan and Motzkin numbers.
They posited conjectures regarding Motzkin numbers modulo $4$ and $8$, which were proved by Eu, Liu, and Yeh~\cite{Eu--Liu--Yeh}.
Congruences for Catalan numbers have also been studied by Liu and Yeh~\cite{Liu--Yeh}, Xin and  Xu~\cite{Xin--Xu}, and Lin~\cite{Lin}.
The techniques used to prove these results depend to some extent on the particular sequences considered, and in some cases the proofs occupy entire papers.
Kauers, Krattenthaler, and M\"uller developed the first systematic methods for producing congruences modulo $2^{\alpha}$ in \cite{Kauers--Krattenthaler--Muller} and modulo $3^{\alpha}$ in \cite{Krattenthaler--Muller} for a large family of differentially algebraic sequences, including many algebraic sequences. As examples they produce automatic proofs of many existing results.

In this paper we show how to discover and prove congruences for algebraic sequences over $\Q(x)$
 in a general fashion --- for any algebraic sequence modulo any prime power.
A natural setting for these results is that  of automatic sequences.
A \emph{$p$-automatic sequence} is a sequence $(a_n)_{n \geq 0}$ on a finite alphabet, where $a_n$ is the output of a finite-state automaton when fed the standard base-$p$ representation of $n$.
We postpone the formal definition until Section~\ref{Notation}.
The following result provides a fundamental link between automaticity and algebraicity; let $\F_p$ denote the finite field of size $p$.

\begin{theorem}[Christol et al.~\cite{CKMR}]\label{Christol} Let 
 $(a_n)_{n \geq 0}$ be a sequence of elements in $\F_p$. Then
 $\sum_{n \geq 0}a_{n}x^{n}$ is algebraic over $\F_{p}(x)$ if and only if 
$(a_{n})_{n \geq 0}$ is $p$-automatic.
\end{theorem}

A proof also appears in \cite[Theorem~12.2.5]{ash}.
It follows immediately that if $(a_n)_{n \geq 0}$ is an algebraic sequence of integers (or, more generally, $p$-adic integers), then $(a_n \bmod p)_{n \geq 0}$ is $p$-automatic, since projecting modulo $p$ a polynomial for which $\sum_{n \geq 0} a_n x^n$ is a root yields a polynomial for which $\sum_{n \geq 0} (a_n \bmod p) x^n$ is a root.

The proof of Theorem~\ref{Christol} is constructive in the sense that, given a polynomial for which $\sum_{n \geq 0} a_n x^n$ is a root, there is an algorithm for producing an automaton that computes $a_n \bmod p$.
Conversely, given such an automaton, there is an algorithm for computing such a polynomial; an example showing the details of this computation appears in \cite[Example~4.2]{Rowland--Yassawi}.

One can also define $p$-automaticity for multidimensional sequences $(a_{n_1, \ldots, n_k})_{n_1, \ldots, n_k \geq 0}$, by feeding, in parallel,  the base-$p$ representations of $n_1, \ldots, n_k$. 
For an introduction, see \cite[Chapter~14]{ash}.
A multidimensional version of Theorem~\ref{Christol} is proved in \cite{Salon}.
 The following generalization of Theorem~\ref{Christol}  was first proved by Christol~\cite{Christol}, for $k=1$, and then later by   Denef and Lipshitz~\cite{Denef--Lipshitz}.
The ring of $p$-adic integers is denoted by $\Z_p$.

\begin{theorem}\label{Denef--Lipshitz_1}
Let  $(a_{n_1, \ldots, n_k})_{n_1, \ldots,n_k \geq 0}$ be a $k$-dimensional sequence of $p$-adic integers such that
\[
	\sum_{n_1, \ldots, n_k \geq 0} a_{n_1, \ldots, n_k} x_1^{n_1}\cdots x_k^{n_k}
\]
is algebraic over $\Z_p(x_1, \ldots, x_k)$, and let $\alpha \geq 1$.
Then $(a_{n_1, \ldots, n_k} \bmod p^\alpha)_{n_1, \ldots, n_k \geq 0}$ is $p$-automatic. 
\end{theorem}

For $k = 1$, it follows from Theorem~\ref{Denef--Lipshitz_1} that  if $(a_n)_{n \geq 0}$ is algebraic, then, given a set $R$ of residue classes modulo $p^\alpha$, the set of words
\[
	\{\text{base-$p$ representation of $n$} \, : \, \text{$a_n \equiv r \mod p^\alpha$ for some $r \in R$}\}
\]
is a regular language.
Given an automaton which computes $a_n \bmod p^\alpha$, an automaton accepting this language can be obtained by setting all states corresponding to an output $r \in R$ as accepting states and all others as rejecting states.
An analogous statement holds for general $k \geq 1$.

Denef and Lipshitz gave two proofs of Theorem~\ref{Denef--Lipshitz_1}.
In this paper we emphasize the extent to which these proofs are constructive.
From each proof we extract an algorithm which, given an appropriate sequence and a prime power $p^\alpha$, outputs a finite automaton that computes terms of the sequence modulo $p^{\alpha}$.

 The first algorithm, which we describe in Section~\ref{easy_algorithm}, is simpler to implement, works for most of the algebraic sequences we considered, and  indeed runs quickly for sequences such as the Catalan and Motzkin numbers modulo small prime powers.
This algorithm in fact applies more generally to diagonals of certain rational power series.
For example, the sequence of Ap\'ery numbers, which has received much attention, is the diagonal of a rational power series but is not algebraic.

  However, for algebraic sequences this algorithm puts requirements on the coefficients of the polynomial satisfied by the generating function.
The second algorithm, described in Section~\ref{hard_algorithm}, applies to all algebraic sequences but in practice is much slower.

In general, neither algorithm produces the automaton with fewest states for a given sequence modulo $p^\alpha$.
However, it is natural to ask, for each of these algorithms, how  the number of states changes as $p$ and $\alpha$ vary. Apart from Remarks~\ref{first estimate} and \ref{second estimate}, we do not address this here, but  Adamczewski and Bell~\cite{Adamczewski--Bell} answered a related question for $\alpha=1$.  In that case, $\sum_{n \geq 0} (a_n \bmod p) x^n$ is algebraic over $\F_p(x)$, and they showed that polynomials for which it is a root have comparable degrees as $p$ varies.
As a consequence, for general $\alpha \geq 1$ they obtain bounds on the degrees of a polynomial for an algebraic sequence modulo $p^\alpha$ \cite[Remark~1.2]{Adamczewski--Bell}.

In Section~\ref{Congruences} we compute, purely mechanically, finite automata for various sequences modulo $p^\alpha$, using the method of Section~\ref{easy_algorithm}.
Any number of congruences can be read off from these automata.
In this way we provide routine proofs of many known results, establish a large number of new congruences for combinatorial sequences, and also prove some conjectures that have not succumbed to other approaches.

Finally, in Section~\ref{Lucas} we consider multidimensional diagonals of rational expressions.
We give general conditions for a Lucas product to exist for the coefficient sequence modulo $p$, and we give a new generalization to prime powers of Lucas' theorem for $\binom{n}{m}$.

We mention that after the present paper appeared in preprint form, Zeilberger and the first author~\cite{Rowland--Zeilberger} gave a method for computing an automaton for $a_n \bmod p^\alpha$, where $a_n$ is the constant term of $P(x)^n Q(x)$ for some Laurent polynomials $P(x), Q(x)$.
The algorithm is similar in many ways to the algorithm we describe in Section~\ref{easy_algorithm} and applies to the same combinatorial sequences of interest.

\subsection{Finite automata and the Cartier operator}\label{Notation}

We now give a formal definition of a finite automaton with output.

\begin{definition}\label{dfao}
A {\em $p$-deterministic finite automaton with output} ($p$-DFAO) is a 6-tuple  $(\mathcal S, \Sigma_{p},\delta, s_1, \mathcal A, \omega)$, where $\mathcal S$ is a finite set of ``states'', $s_1 \in \mathcal S$ is the {\em initial state}, $\Sigma_p=\{0,1, \ldots, p-1\}$,
$\mathcal A$ is a finite alphabet,
$\omega:\mathcal S\rightarrow \mathcal A$ is the \emph{output function}, and $\delta:\mathcal S\times \Sigma_{p}\rightarrow \mathcal S$ is the {\em transition function}.
\end{definition}

The function $\delta$ extends in a natural way to the domain $\mathcal S \times \Sigma_p^+$, where $\Sigma_p^+$ is the set of nonempty words on the alphabet $\Sigma_p$.
Namely, define $\delta(s, n_l \cdots n_1 n_0) \colonequal \delta(\delta(s, n_0), n_l \cdots n_1)$ recursively.
This allows us to feed the standard base-$p$ representation $n_l \cdots n_1 n_0$ of an integer $n$ into an automaton.
Our convention is that we read the base-$p$ representation beginning with the \emph{least significant digit}.
(Recall that the standard base-$p$ representation of $0$ is the empty word.)

\begin{definition}
A sequence $(a_n)_{n \geq 0}$ of elements in $\mathcal A$ is {\em $p$-automatic} if there is a $p$-DFAO  
$(\mathcal S,\Sigma_{p},\delta, s_1, \mathcal A, \omega)$ such that $a_{n} = \omega(\delta (s_1, n_l \cdots n_1 n_0))$ for all $n \geq 0$, where $n_l \cdots n_1 n_0$ is the standard base-$p$ representation of $n$.
\end{definition}

In this article our alphabet is $\mathcal A = \Z/(p^\alpha \Z)$, where $p$ is a prime and $\alpha \geq 1$.

\begin{example}
Consider the following automaton for $p = 2$ and $\alpha = 2$.
Each of the six states is represented by a vertex, labeled with its output under $\omega$.
Edges between vertices illustrate $\delta$.
The unlabeled edge points to the initial state.
\begin{center}
	\scalebox{.8}{\vcentergraphics{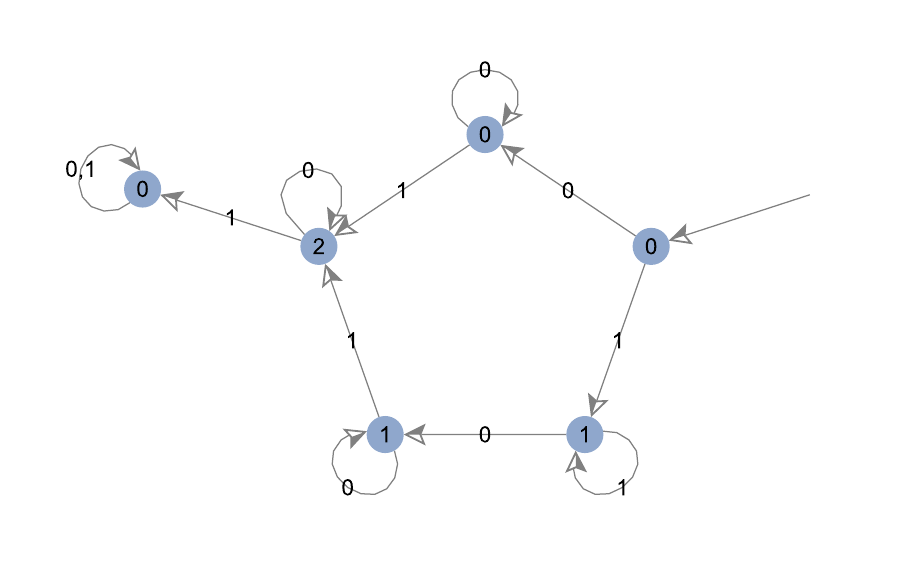}}
\end{center}
The $2$-automatic sequence produced by this automaton is
\[
	(a_n)_{n \geq 0} = 0, 1, 2, 1, 2, 2, 0, 1, 2, 2, 0, 2, 0, 0, 0, 1, \dots.
\]
We will see in Section~\ref{Catalan numbers} that for $n \geq 1$ this is the sequence of Catalan numbers modulo $4$.
\end{example}

\begin{definition}
The {\em $p$-kernel} of a sequence $(a_{n})_{n\geq 0}$ is the collection of sequences
\[\ker_p((a_n)_{n \geq 0}) \colonequal \{(a_{p^{e}n+j})_{n\geq 0} \, : \, e\geq 0, \, 0\leq j\leq p^{e}-1\}.\]
\end{definition}

If $\mathcal{A}$ is a ring, we let $\mathcal{A}[x_1, \dots, x_k]$ and $\mathcal{A}\llbracket x_1, \dots, x_k \rrbracket$ denote the sets of polynomials and formal power series, respectively, in variables $x_1, \dots, x_k$ with coefficients in $\mathcal{A}$.
The power series $f(x_1, \ldots,x_k)\in \mathcal{A}\llbracket x_1, \ldots, x_k \rrbracket$ is {\em algebraic}
if there exists a nonzero polynomial $P(x_1, \ldots, x_k, y) \in \mathcal{A}[x_1, \ldots, x_k, y]$ such that $P(x_1, \ldots, x_k, f(x_1, \ldots, x_k)) = 0$.

By identifying a sequence with its generating function, we extend the notion of the $p$-kernel to formal power series. Namely, if $f(x)=\sum_{n \geq 0} a_n x^n$, then
\[
	\ker_p(f(x)) \colonequal \left\{\sum_{n \geq 0} a_{p^{e}n+j} x^n \, : \, e\geq 0, \, 0\leq j\leq p^{e}-1\right\}.
\]
The set $\ker_p(f) \bmod p^\alpha$ is the set $\ker_p(f)$ with each element projected modulo $p^\alpha$.

The Cartier operator provides a standard way to access elements of the $p$-kernel.

\begin{definition}
Fix $p$, and let $(d_1, \dots, d_k) \in \{0, 1, \dots, p-1\}^k$.
The \emph{Cartier operator} $\Lambda_{d_1, \dots, d_k}$ is the map on $\mathcal{A}\llbracket x_1, \dots, x_k \rrbracket$ defined by
\[
	\Lambda_{d_1, \dots, d_k} \left(\sum_{n_1, \dots, n_k \geq 0} a_{n_1, \dots, n_k} x_1^{n_1} \cdots x_k^{n_k}\right)
	\colonequal \sum_{n_1, \dots, n_k \geq 0} a_{p n_1 + d_1, \dots, p n_k + d_k} x_1^{n_1} \cdots x_k^{n_k}.
\]
\end{definition}

Equivalently,
\[
	\Lambda_{d_1, \dots, d_k} \left(\sum_{n_1, \dots, n_k \geq 0} a_{n_1, \dots, n_k} x_1^{n_1} \cdots x_k^{n_k}\right)
	= \sum_{\substack{n_1 \equiv d_1 \mod p \vspace{-.2cm} \\ \vdots \\ n_k \equiv d_k \mod p}} a_{n_1, \dots, n_k} x_1^{\lfloor n_1/p \rfloor} \cdots x_k^{\lfloor n_k/p \rfloor}.
\]
Note that, in one variable, $f(x) = \sum_{d=0}^{p-1} x^{d} \lam_d (f) (x^p)$, and moreover if $f(x) \in \F_p\llbracket x \rrbracket$, then  $f(x) = \sum_{d=0}^{p-1} x^{d} (\lam_d (f(x)))^p$. Note also  that 
\[
	\lam_{d_l} \circ \cdots \circ \lam_{d_1} \circ \lam_{d_0} (f)
	= \sum_{n \geq 0} a_{p^{l+1} n + (p^l d_l + \cdots + p^1 d_1 + p^0 d_0)} x^n,
\]
so that
\[
	\ker_p(f) = \{f\} \cup \{\lam_{d_l} \circ \dots \circ \lam_{d_1} \circ \lam_{d_0} (f) \, : \,  l\geq 0, \, 0\leq d_j \leq p-1 \mbox{ for each } j\}.
\]

The following classical result can be found in \cite[Proposition~V.3.3]{ei} and \cite[Theorem~6.6.2]{ash}.
Our methods use this theorem and its proof heavily, so we include a proof.

\begin{theorem}\label{Eilenberg}
Let 
 $(a_n)_{n \geq 0}$ be a sequence of elements from  a finite alphabet $\mathcal A$. Then
the $p$-kernel of $(a_{n})_{n\geq 0}$ is finite if and only if 
$(a_{n})_{n \geq 0}$ is $p$-automatic. 
\end{theorem}

\begin{proof}
Note that we need not assume $p$ is prime, so the theorem holds more generally.

Suppose the $p$-kernel of $(a_{n})_{n\geq 0}$ is finite.
Build an automaton as follows.
Let $\mathcal S$ be the $p$-kernel of $(a_{n})_{n\geq 0}$, and designate the sequence $(a_{n})_{n\geq 0}$ itself to be the initial state $s_1 \in \mathcal S$.
By identifying a sequence with its generating function, we can define $\Lambda_d$ on $\mathcal S$.
For all $s \in \mathcal S$ and $d \in \Sigma_p$, let $\delta(s, d) = \Lambda_d(s)$.
Finally, for each $s \in \mathcal S$ let $\omega(s)$ be the first term of $s$.
Then we claim the automaton $(\mathcal S, \Sigma_{p},\delta, s_1, \mathcal A, \omega)$ outputs $a_n$ when fed the base-$p$ digits of $n$.
Clearly this true for $n = 0$, since $\omega(s_1) = a_0$.
For $n \geq 1$ we have
\begin{align*}
	\omega(\delta(s_1, n_l \cdots n_1 n_0))
	&= \omega(\delta( \cdots \delta(\delta(s_1, n_0), n_1) \cdots, n_l)) \\
	&= \omega(\Lambda_{n_l} \circ \dots \circ \Lambda_{n_1} \circ \Lambda_{n_0}(s_1)) \\
	&= a_n.
\end{align*}

Conversely, let $(\mathcal S, \Sigma_{p},\delta, s_1, \mathcal A, \omega)$ be an automaton that outputs $a_n$ when fed the base-$p$ digits of $n$.
For a sequence $(a_{p^{e}n+j})_{n\geq 0}$ in the $p$-kernel of $(a_n)_{n \geq 0}$, write $j = d_{e-1} \cdots d_1 d_0$ in base $p$, and let $s_{e,j} = \delta(s_1, d_{e-1} \cdots d_1 d_0) \in \mathcal S$.
Then the automaton $(\mathcal S, \Sigma_{p},\delta, s_{e,j}, \mathcal A, \omega)$ outputs $a_{p^{e}n+j}$ when fed the base-$p$ digits of $n$.
This gives an injection from the $p$-kernel of $(a_n)_{n \geq 0}$ to $\mathcal S$, so the finiteness of the $p$-kernel now follows from the finiteness of $\mathcal S$.
\end{proof}

From the proof of Theorem~\ref{Eilenberg} it follows that if $(a_{n})_{n \geq 0}$ is $p$-automatic, then relationships between the elements of its $p$-kernel can be explicitly read off from any automaton that computes $(a_{n})_{n \geq 0}$ (reading least significant digit first).
Coupled with Theorem~\ref{Christol}, this implies that, given a polynomial $P(x,y) \in \F_p[x,y]$ such that $P(x,\sum_{n \geq 0}a_{n}x^{n})=0$, one can compute the $p$-kernel of $(a_n)_{n\geq 0}$.

The following proposition is highly useful.
It shows that we can pull certain power series out of the Cartier operator when working modulo $p^\alpha$.

\begin{proposition}\label{Cartier extraction}
Let $x = (x_1, \dots, x_k)$.
Let $f(x), g(x) \in \Z_p\llbracket x \rrbracket$ be formal power series in $k$ variables, and let $r \in \{0, \dots, p-1\}^k$.
Then
\[
	\Lambda_r(g(x) \cdot f(x)^{p^\alpha}) \equiv \Lambda_r(g(x)) \cdot f(x)^{p^{\alpha - 1}} \mod p^\alpha.
\]
\end{proposition}

\begin{proof}
If $a, b \in \Z_p$ such that $a \equiv b \mod p$, then $a^{p^{\alpha-1}} \equiv b^{p^{\alpha-1}} \mod p^\alpha$.
Since $f(x)^p \equiv f(x^p) \mod p$, it follows that $f(x)^{p^\alpha} \equiv f(x^p)^{p^{\alpha-1}} \mod p^\alpha$.
One verifies that $\Lambda_r(g(x) \cdot h(x^p)) = \Lambda_r(g(x)) \cdot h(x)$.
Therefore
\begin{align*}
	\Lambda_r\left(g(x) \cdot f(x)^{p^\alpha}\right)
	&\equiv \Lambda_r\left(g(x) \cdot f(x^p)^{p^{\alpha-1}}\right) \mod p^\alpha \\
	&= \Lambda_r(g(x)) \cdot f(x)^{p^{\alpha-1}}. \qedhere
\end{align*}
\end{proof}

\section{Automata for diagonals of rational power series}\label{easy_algorithm}

In this section we give an algorithm for computing automata for sequences, modulo $p^\alpha$, that are diagonals of certain rational power series.  This includes many algebraic sequences.
The approach is based on a proof of Theorem~\ref{Denef--Lipshitz_1} by Denef and Lipshitz~\cite[Remark~6.6]{Denef--Lipshitz}.
Part of their argument \cite[Theorem~6.2]{Denef--Lipshitz} is nonconstructive.
Therefore, applying our algorithm to an algebraic power series requires a polynomial of a certain form.
However, we were able to apply the algorithm to nearly all combinatorial sequences we considered.

The \emph{diagonal} of a formal power series is
\[
	\mathcal{D} \left(\sum_{n_1, \dots, n_k \geq 0} a_{n_1, \dots, n_k} x_1^{n_1} \cdots x_k^{n_k}\right)
	\colonequal \sum_{n \geq 0} a_{n, \dots, n} x^n.
\]

\begin{theorem}\label{diagonal closure}
Let $R(x_1,\ldots, x_k)$ and $Q(x_1, \ldots,x_k)$ be polynomials in $\Z_p[x_1, \ldots, x_k]$ such that $Q(0,\ldots,0) \nequiv 0 \mod p$, and let $\alpha\geq 1$.
Then the coefficient sequence of
\[
	\mathcal{D} \left(\frac{R(x_1, \ldots, x_k)}{Q(x_1, \ldots,x_k)}\right) \bmod p^\alpha
\]
is $p$-automatic.
\end{theorem}

\begin{proof}
Since $Q(0,\ldots,0) \nequiv 0 \mod p$, we can expand $R(x_1,\ldots, x_k)/Q(x_1, \ldots,x_k)$ as a power series whose coefficients are $p$-adic integers.

Let $\mathcal{A} = \Z_p/(p^\alpha\Z_p)$.
By Proposition~\ref{Cartier extraction}, for $s(x_1, \ldots, x_k) \in \Z_p[x_1,\ldots, x_k]$ we have
\begin{align*}
	\Lambda_{d_1, \ldots, d_k}\left(\frac{s(x_1, \ldots, x_k)}{Q(x_1, \ldots, x_k)^{p^{\alpha-1}}}\right)
	&= \Lambda_{d_1, \ldots, d_k}\left(\frac{s(x_1, \ldots, x_k) \cdot Q(x_1, \ldots, x_k)^{p^\alpha - p^{\alpha-1}}}{Q(x_1, \ldots, x_k)^{p^\alpha}}\right) \\
	&\equiv \frac{\Lambda_{d_1, \ldots, d_k}\left(s(x_1, \ldots, x_k) \cdot Q(x_1, \ldots, x_k)^{p^\alpha - p^{\alpha-1}}\right)}{Q(x_1, \ldots, x_k)^{p^{\alpha-1}}} \mod p^\alpha.
\end{align*}
Since the denominator $Q(x_1, \ldots, x_k)^{p^{\alpha-1}}$ appears both in the initial and final expression, we consider the map $\mu_{d_1, \ldots, d_k}$ from $\mathcal{A}[x_1, \ldots, x_k]$ to itself given by
\[
	\mu_{d_1, \ldots, d_k}(s(x_1, \ldots, x_k))
	\colonequal
	\Lambda_{d_1, \ldots, d_k}\left(s(x_1, \ldots, x_k) \cdot Q(x_1, \ldots, x_k)^{p^\alpha - p^{\alpha-1}}\right) \bmod p^\alpha.
\]
Let $\deg s(x_1, \ldots, x_k) \colonequal \max_{1 \leq i \leq k} \deg_{x_i} s(x_1, \ldots, x_k)$ be the degree of a polynomial $s(x_1, \ldots, x_k)$.
The degree of $\mu_{d_1, \ldots, d_k}(s(x_1, \ldots, x_k))$ is at most
\[
	\frac{1}{p} \left(\deg s(x_1, \ldots, x_k) + (p^\alpha - p^{\alpha-1}) \deg Q(x_1, \ldots, x_k)\right).
\]
The fixed point of the map
\[
	m \mapsto \frac{1}{p} \left(m + (p^\alpha - p^{\alpha-1}) \deg Q(x_1, \ldots, x_k)\right)
\]
is $p^{\alpha-1} \deg Q(x_1, \ldots, x_k)$.
Let
\[
	m = \max\left\{\deg \left(R(x_1, \ldots, x_k) \cdot Q(x_1, \ldots, x_k)^{p^{\alpha-1} - 1}\right), p^{\alpha-1} \deg Q(x_1, \ldots, x_k)\right\},
\]
and let $\mathcal{S}$ be the set of all polynomials in $\mathcal{A}[x_1, \ldots, x_k]$ with degree at most $m$.
Then $\left(R(x_1, \ldots, x_k) \cdot Q(x_1, \ldots, x_k)^{p^{\alpha-1} - 1} \bmod p^\alpha\right) \in \mathcal{S}$, and if $s(x_1, \ldots, x_k) \in \mathcal{S}$ then $\mu_{d_1, \ldots, d_k}(s(x_1, \ldots, x_k)) \in \mathcal{S}$.
Therefore $\mathcal{D} (R(x_1, \ldots, x_k)/Q(x_1, \ldots, x_k) \bmod p^\alpha) \in \mathcal{D}\left(\mathcal{S}/Q(x_1, \ldots, x_k)^{p^{\alpha-1}}\right)$, and $\mathcal{S}$ is closed under $\mu_{d_1, \ldots, d_k}$.
Since
\[
	\Lambda_d\left( \mathcal{D}\left( \frac{s(x_1, \ldots, x_k)}{Q(x_1, \ldots, x_k)^{p^{\alpha-1}}} \right) \right)
	\equiv \mathcal{D}\left( \frac{\mu_{d,\ldots, d}(s(x_1, \ldots, x_k))}{Q(x_1, \ldots, x_k)^{p^{\alpha-1}}} \right) \mod p^\alpha,
\]
the finiteness of $\ker_p(\mathcal{D} (R(x_1, \ldots, x_k)/Q(x_1, \ldots, x_k) \bmod p^\alpha))$ now follows from the finiteness of $\mathcal{S}$.
By Theorem~\ref{Eilenberg}, the sequence of coefficients is $p$-automatic.
\end{proof}

The relationships between the elements of the $p$-kernel of $(a_n)_{n \geq 0}$ encode a finite automaton for $(a_n)_{n \geq 0}$ in which each state corresponds to an element of the $p$-kernel and where $p$ outgoing edges from a state point to its images under $\Lambda_d$.
Therefore we see from the proof of Theorem~\ref{diagonal closure} that an automaton for the coefficients of $\mathcal{D} (R(x_1, \ldots, x_k)/Q(x_1, \ldots, x_k))$ modulo $p^\alpha$ can be computed as follows.

We perform all arithmetic in $\mathcal{A} = \Z/(p^\alpha\Z) \cong \Z_p/(p^\alpha\Z_p)$.
Multiply $R(x_1, \ldots, x_k)$ and $Q(x_1, \ldots, x_k)$ by $Q(0, \ldots, 0)^{-1}$, so that we may assume $Q(0, \ldots, 0) = 1$.
Let the initial state be $R(x_1, \ldots, x_k) \cdot Q(x_1, \ldots, x_k)^{p^{\alpha-1} - 1} \in \mathcal{S}$.
Apply each $\mu_{d, \ldots, d}$, for $0 \leq d \leq p - 1$, to the initial state, obtaining $p$ elements of $\mathcal{S}$.
Some of these polynomials may coincide with the initial state, in which case we have already computed their images under $\mu_{d, \ldots, d}$.
For the polynomials whose images under $\mu_{d, \ldots, d}$ have not yet been computed, compute them.
Iterate, and stop when all images have been computed.
Draw an edge labeled $d$ from $s(x_1, \dots, x_k)$ to $t(x_1, \dots, x_k)$ if $\mu_{d,\dots,d}(s(x_1, \dots, x_k)) = t(x_1, \dots, x_k)$.
The automaton's output corresponding to each state $s(x_1, \ldots, x_k)$ is the constant term of the series $s(x_1, \ldots, x_k) / Q(x_1, \ldots, x_k)^{p^{\alpha-1}}$; since $Q(0, \ldots, 0) = 1$, this constant term is $s(0, \ldots, 0)$.

Algorithm~\ref{slow diagonal algorithm} contains a more formal description.
Many of our applications, to be discussed shortly, will only require rational expressions in two variables, so for concreteness Algorithm~\ref{slow diagonal algorithm} is written for a bivariate expression $R(x,y)/Q(x,y)$.
The input consists of a prime $p$, an integer $\alpha \geq 1$, and polynomials $R(x, y), Q(x, y) \in \mathcal{A}[x, y]$ such that $Q(0,0) = 1$.
Since all arithmetic is performed in $\mathcal{A} = \Z/(p^\alpha\Z)$, $R(x,y)$ and $Q(x,y)$ can be given as polynomials with coefficients in this ring, even if they started as polynomials with integer or $p$-adic integer coefficients.

The output of Algorithm~\ref{slow diagonal algorithm} is a finite automaton represented as a $6$-tuple as in Definition~\ref{dfao}.
We construct the functions $\delta$ and $\omega$ one state at a time, so it will be convenient to represent these functions as sets of pairs.
The pair $n \to a$ in the set $\omega$ represents the value $\omega(n) = a$, the output corresponding to state $n$.
The pair $(n, d) \to i$ in the set $\delta$ represents the value $\delta(n, d) = i$, which corresponds to a directed edge from state $n$ to state $i$ that is labeled by $d$.
We maintain $n$ as the index of the state we are currently examining and $m$ as the total number of states.

\begin{algorithm}\label{slow diagonal algorithm}
\SetArgSty{}
\DontPrintSemicolon
\KwIn{$(R(x, y), Q(x, y), p, \alpha) \in \mathcal{A}[x, y] \times \mathcal{A}[x, y] \times \mathbb{P} \times \Z_{\geq 1}$ with $Q(0,0) = 1$}

$\delta \leftarrow \varnothing$\;

$m \leftarrow 1$\;

$s_1(x,y) \leftarrow R(x, y) \cdot Q(x, y)^{p^{\alpha-1} - 1}$\;

$n \leftarrow 1$\;

\While{$n \leq m$}{

	\For{$d \in \{0, 1, \dots, p-1\}$}{

		$s(x,y) \leftarrow \Lambda_{d,d}\left(s_n(x, y) \cdot Q(x, y)^{p^\alpha - p^{\alpha-1}}\right)$\;

		\eIf{$s(x, y) \in \{s_1(x, y), s_2(x, y), \dots, s_m(x, y)\}$}{
			$\delta \leftarrow \delta \cup \{(n, d) \to i\}$, where $s(x, y) = s_i(x, y)$\;
		}{
			$m \leftarrow m + 1$\;
			$s_m(x, y) \leftarrow s(x, y)$\;
			$\delta \leftarrow \delta \cup \{(n, d) \to m\}$\;
		}

	}

	$n \leftarrow n + 1$\;

}

$\omega \leftarrow \{1 \to s_1(0, 0), 2 \to s_2(0, 0), \dots, m \to s_m(0, 0)\}$\;

\Return $(\{1, 2, \dots, m\}, \Sigma_p, \delta, 1, \mathcal A, \omega)$\;
\caption{Computing an automaton for the diagonal of a bivariate rational expression $R(x, y)/Q(x, y)$ modulo $p^\alpha$.}
\end{algorithm}

\begin{remark} \label{first estimate}We can give a crude upper bound on the number of states in the automaton output by Algorithm~\ref{slow diagonal algorithm} by computing the number of polynomials in $\mathcal{S}$. Let $L= \max\{\deg R(x_1, \dots, x_k), \deg Q(x_1, \dots, x_k) \}$.
In the notation of the proof of Theorem~\ref{diagonal closure}, we have $m \leq   p^{\alpha-1} L$, so $|\mathcal{S}|\leq p^{\alpha(p^{\alpha-1}L+1)^k}$. Since the state set can be injected into $\mathcal{S}$, this gives us an upper bound for the number of states, although in practice this appears to be a vast overestimate.
The running time of Algorithm~\ref{slow diagonal algorithm} is linear in the number of states of the automaton; consequently we do not have good bounds on the running time.
\end{remark}

One of our primary uses of Theorem~\ref{diagonal closure} will be in conjunction with the following result of Furstenberg~\cite[Proposition~2]{Fur}.
Given an appropriate polynomial for which a power series $f(x)$ is a root, it constructs a rational expression of which $f(x)$ is the diagonal.
A straightforward generalization to multivariate power series was given by Denef and Lipshitz~\cite[Lemma~6.3]{Denef--Lipshitz}.

\begin{proposition}\label{Furstenberg}
Let $P(x, y) \in \Z_p[x, y]$ such that $\frac{\partial P}{\partial y}(0, 0) \neq 0$.
If $f(x) = \sum_{n \geq 0} a_n x^n \in \Z_p\llbracket x \rrbracket$ is a power series such that $a_0 = 0$ and $P(x, f(x)) = 0$, then
\[
	f(x) = \mathcal{D}\left( \frac{y^2 \frac{\partial P}{\partial y}(x y, y)}{P(x y, y)} \right).
\]
\end{proposition}

Under the conditions of Proposition~\ref{Furstenberg}, it follows from $P(x, f(x)) = 0$ and $f(0) = 0$ that $P(0, 0) = 0$.
Therefore, we can factor $y$ out of $P(x y, y)$, and $P(x y, y) / y$ has a nonzero constant term.
If
\begin{equation}\label{partial derivative}
	\tfrac{\partial P}{\partial y}(0, 0) \nequiv 0 \mod p,
\end{equation}
one can compute an automaton for $a_n \bmod p^\alpha$ by letting $c = \left(\frac{\partial P}{\partial y}(0,0)\right)^{-1} \bmod p^\alpha$ and executing Algorithm~\ref{slow diagonal algorithm} on the input
\[
	\left(c \cdot y \cdot \tfrac{\partial P}{\partial y}(x y, y), \, c \cdot P(x y, y)/y, \, p, \, \alpha\right).
\]

If $a_0 \neq 0$ for a given power series $f(x) = \sum_{n \geq 0} a_n x^n$ whose coefficients we would like to determine modulo $p^\alpha$, we must instead consider $f(x) - a_0$ or another modification.
The reader may now wish to turn to Section~\ref{Congruences}, which contains many examples.

As written, the polynomial arithmetic performed in Algorithm~\ref{slow diagonal algorithm} is quite slow for large-degree polynomials.
One computational task that we repeat many times is multiplication by $Q(x,y)^{p^\alpha - p^{\alpha-1}}$.
A simple observation to improve speed is that we should only expand $T(x,y) \colonequal Q(x,y)^{p^\alpha - p^{\alpha-1}}$ once.

Next, observe that each $\Lambda_{d,d}$ discards all monomials $a_{n,m} x^n y^m$ in the product $s(x, y) \cdot T(x, y)$ for which $n \nequiv m \mod p$.
Such monomials represent $(p-1)/p$ of all monomials in this product, so we will significantly reduce the number of arithmetic operations performed if we can avoid computing them in the first place.

This can be accomplished by binning the monomials $a_{n,m} x^n y^m$ in both $s(x,y)$ and $T(x,y)$ according to the difference $n - m$.
Define an operator $\Delta_r$ by
\[
	\Delta_r\left(\sum_{n,m \geq 0} a_{n,m} x^n y^m\right) \colonequal \sum_{n - m \equiv r \mod p} a_{n,m} x^n y^m.
\]
Then the sum
\[
	\sum_{r=0}^{p-1} \Delta_r(s(x,y)) \cdot \Delta_{p-r}(T(x,y))
\]
is the sum of the monomials in $s(x,y) \cdot T(x,y)$ whose exponents are congruent to each other modulo $p$.
Algorithm~\ref{fast diagonal algorithm} incorporates this improvement.
The \textit{Mathematica} implementation used to compute the results in Section~\ref{Congruences} is available from the web site of the first author.

\begin{algorithm}\label{fast diagonal algorithm}
\SetArgSty{}
\DontPrintSemicolon
\KwIn{$(R(x, y), Q(x, y), p, \alpha) \in \mathcal{A}[x, y] \times \mathcal{A}[x, y] \times \mathbb{P} \times \Z_{\geq 1}$ with $Q(0,0) = 1$}

$\delta \leftarrow \varnothing$\;

$m \leftarrow 1$\;

$s_1(x,y) \leftarrow R(x, y) \cdot Q(x, y)^{p^{\alpha-1} - 1}$\;

$n \leftarrow 1$\;

$T(x,y) \leftarrow Q(x,y)^{p^\alpha - p^{\alpha-1}}$\;

\For{$r \in \{0, 1, \dots, p-1\}$}{
	$T_r(x,y) \leftarrow \Delta_r(T(x,y))$\;
}

\While{$n \leq m$}{

	\For{$d \in \{0, 1, \dots, p-1\}$}{

		$s(x,y) \leftarrow \Lambda_{d,d}\left(\sum_{r=0}^{p-1} \Delta_r(s_n(x,y)) \cdot T_{p-r}(x,y)\right)$\;

		\eIf{$s(x, y) \in \{s_1(x, y), s_2(x, y), \dots, s_m(x, y)\}$}{
			$\delta \leftarrow \delta \cup \{(n, d) \to i\}$, where $s(x, y) = s_i(x, y)$\;
		}{
			$m \leftarrow m + 1$\;
			$s_m(x, y) \leftarrow s(x, y)$\;
			$\delta \leftarrow \delta \cup \{(n, d) \to m\}$\;
		}

	}

	$n \leftarrow n + 1$\;

}

$\omega \leftarrow \{1 \to s_1(0, 0), 2 \to s_2(0, 0), \dots, m \to s_m(0, 0)\}$\;

\Return $(\{1, 2, \dots, m\}, \Sigma_p, \delta, 1, \mathcal A, \omega)$\;
\caption{Computing an automaton for the diagonal of a rational expression $R(x, y)/Q(x, y)$ modulo $p^\alpha$, using fewer operations than Algorithm~\ref{slow diagonal algorithm}.}
\end{algorithm}

Finally, note that all $\Delta_r(s(x,y))$ for $r \in \{0, 1, \dots, p-1\}$ can be computed with one pass through $s(x,y)$ rather than $p$ passes.
Similarly, the images of a polynomial under all $\Lambda_{d,d}$ for $d \in \{0, 1, \dots, p-1\}$ can be computed with one pass through the polynomial.

We mention that a different map $\mu_{d_1, \dots, d_k}$ could have been used in the proof of Theorem~\ref{diagonal closure}, and this map yields a slightly different algorithm.
Namely, we have
\begin{align*}
	\Lambda_{d_1, \ldots, d_k}\left(\frac{s(x_1, \ldots, x_k)}{Q(x_1, \ldots, x_k)^{p^\alpha}}\right)
	&\equiv \frac{\Lambda_{d_1, \ldots, d_k}(s(x_1, \ldots, x_k))}{Q(x_1, \ldots, x_k)^{p^{\alpha-1}}} \mod p^\alpha \\
	&= \frac{\Lambda_{d_1, \ldots, d_k}(s(x_1, \ldots, x_k)) \cdot Q(x_1, \ldots, x_k)^{p^\alpha - p^{\alpha-1}}}{Q(x_1, \ldots, x_k)^{p^\alpha}}.
\end{align*}
(Note that in~\cite[Remark~6.6]{Denef--Lipshitz} the exponent in the numerator of this expression is incorrect.)
In two variables, the map suggested by this identity is
\[
	s(x, y) \mapsto \Lambda_{d,d}(s(x,y)) \cdot Q(x,y)^{p^\alpha - p^{\alpha-1}} \bmod p^\alpha.
\]
Since the exponent in the denominator is $p^\alpha$ rather than $p^{\alpha - 1}$, this map results in a higher maximum degree $m$, so the states $s(x,y)$ have higher degree in this algorithm and are consequently slower to compute with.
The benefit of this algorithm is that it sometimes produces automata with fewer states relative to Algorithms~\ref{slow diagonal algorithm} and \ref{fast diagonal algorithm}.
This may be related to the fact that $\Lambda_{d,d}$ applies to $s(x, y)$ before multiplication by $Q(x,y)^{p^\alpha - p^{\alpha-1}}$, so monomials $a_{n,m} x^n y^m$ in $s(x, y)$ for which $n \nequiv m \mod p$ are discarded.
When these monomials are omitted from each state, two states which previously differed only by monomials with incongruent exponents collapse into a single state.

\section{Congruences}\label{Congruences}

In this section we consider a number of combinatorial sequences and give congruences that were proved (and in most cases also discovered) by computing a finite automaton for the sequence modulo $p^\alpha$ using Algorithm~\ref{fast diagonal algorithm}.
Details of the computations appear on the web site of the first author\footnote{\url{http://thales.math.uqam.ca/~rowland/packages.html\#IntegerSequences} as of this writing.}.

\subsection{Catalan numbers}\label{Catalan numbers}

Let $C(n) = \frac{1}{n+1} \binom{2n}{n}$.
The sequence $C(n)_{n \geq 0} = 1, 1, 2, 5, 14, 42, 132, 429, \dots$ of Catalan numbers~\seq{A000108} is arguably the most important sequence in combinatorics.
Catalan numbers were studied from an arithmetic perspective by Alter and Kubota~\cite{Alter--Kubota}.
Even before their work, the value of $C(n) \bmod 2$ was already known.

\begin{theorem}\label{Catalan mod 2}
For all $n \geq 0$, $C(n)$ is odd if and only if $n = 2^k - 1$ for some $k \geq 0$.
\end{theorem}

Eu, Liu, and Yeh~\cite{Eu--Liu--Yeh} determined the value of $C(n)$ modulo $4$ and, more generally, modulo $8$.
Xin and Xu~\cite{Xin--Xu} provided shorter proofs of these results.
Modulo $4$, the Catalan numbers have the following forbidden residue.

\begin{theorem}[Eu--Liu--Yeh]\label{Catalan mod 4}
For all $n \geq 0$, $C(n) \nequiv 3 \mod 4$.
\end{theorem}

Liu and Yeh~\cite{Liu--Yeh} determined $C(n)$ modulo $16$ and $64$, written using piecewise functions with many cases; we argue that finite automata provide more natural notation.

Using the method of Section~\ref{easy_algorithm}, we can generate and prove such results automatically.
Kauers, Krattenthaler, and M\"{u}ller~\cite[Section~5]{Kauers--Krattenthaler--Muller} had a similar goal, and they showed how to produce congruences for $C(n)$ modulo an arbitrary power of $2$, encoded as formal power series rather than finite automata.

\begin{figure}
	\begin{center}
		\scalebox{.8}{\vcentergraphics{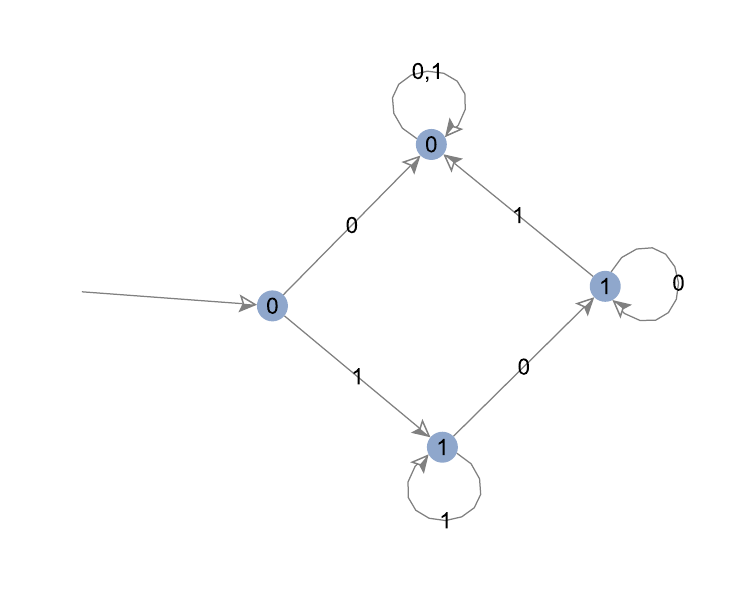}}
		\scalebox{.8}{\vcentergraphics{Catalan4}}
		\caption{Automata that compute Catalan numbers modulo $2$ (left) and $4$ (right).}
		\label{Catalan 2 and 4}
	\end{center}
\end{figure}

The generating function $z = \sum_{n \geq 0} C(n) x^n$ for the Catalan numbers satisfies
\[
	x z^2 - z + 1 = 0.
\]
Since $C(0) \neq 0$, the series $\sum_{n \geq 0} C(n) x^n$ does not satisfy the conditions of Proposition~\ref{Furstenberg}.
To remedy this, we modify the first term and instead consider the series $y = 0 + \sum_{n \geq 1} C(n) x^n$, which satisfies
\[
	x (y + 1)^2 - (y + 1) + 1 = 0.
\]
Writing this equation as
\[
	P(x, y) \colonequal x y^2 + (2 x - 1) y + x = 0,
\]
we see that $\frac{\partial P}{\partial y}(0, 0) = -1 \nequiv 0 \mod 2$, satisfying Equation~\eqref{partial derivative}.
By Proposition~\ref{Furstenberg}, $\sum_{n \geq 1} C(n) x^n$ is the diagonal of
\[
	\frac{y (2 x y^2 + 2 x y - 1)}{x y^2 + 2 x y + x - 1}.
\]
Reducing the coefficients modulo $2$ and executing Algorithm~\ref{fast diagonal algorithm} on the input
\[
	\left(y, \, x y^2 + x + 1, \, 2, \, 1\right)
\]
yields the automaton on the left in Figure~\ref{Catalan 2 and 4}, whose four states are represented by the polynomials $y, 0, y + 1, 1$.
Remember that this automaton (like the others we compute below) outputs $0$ for $n = 0$; its output is $C(n) \bmod 2$ only for $n \geq 1$.
Modifying finitely many terms of an automatic sequence produces another automatic sequence, so it is possible to modify the automaton so that it outputs $C(0) \bmod 2$ for $n = 0$, repairing the initial term, although we do not undertake this here.

An inspection of this automaton shows that the input string $111 \cdots 1$ outputs $1$.
Moreover, since the most significant digit in the binary representation of $n$ is $1$ for all $n \geq 1$, these are the only input strings that output $1$.
We have therefore proved Theorem~\ref{Catalan mod 2}.

Automata for higher powers of $2$ can be computed similarly.
To prove Theorem~\ref{Catalan mod 4}, we compute an automaton for $C(n) \bmod 4$, obtaining the automaton on the right in Figure~\ref{Catalan 2 and 4}.
It contains six states, none of which correspond to the output $3$.
It remains to check $n = 0$, for which $C(0) = 1 \nequiv 3 \mod 4$.

\begin{figure}
	\begin{center}
		\scalebox{.75}{\vcentergraphics{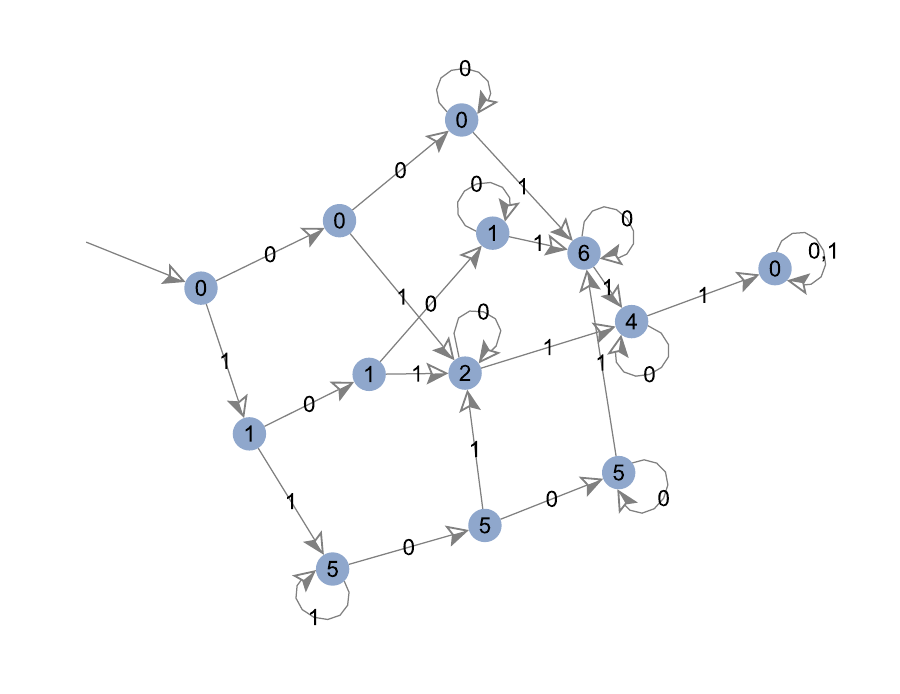}}
		\scalebox{.75}{\vcentergraphics{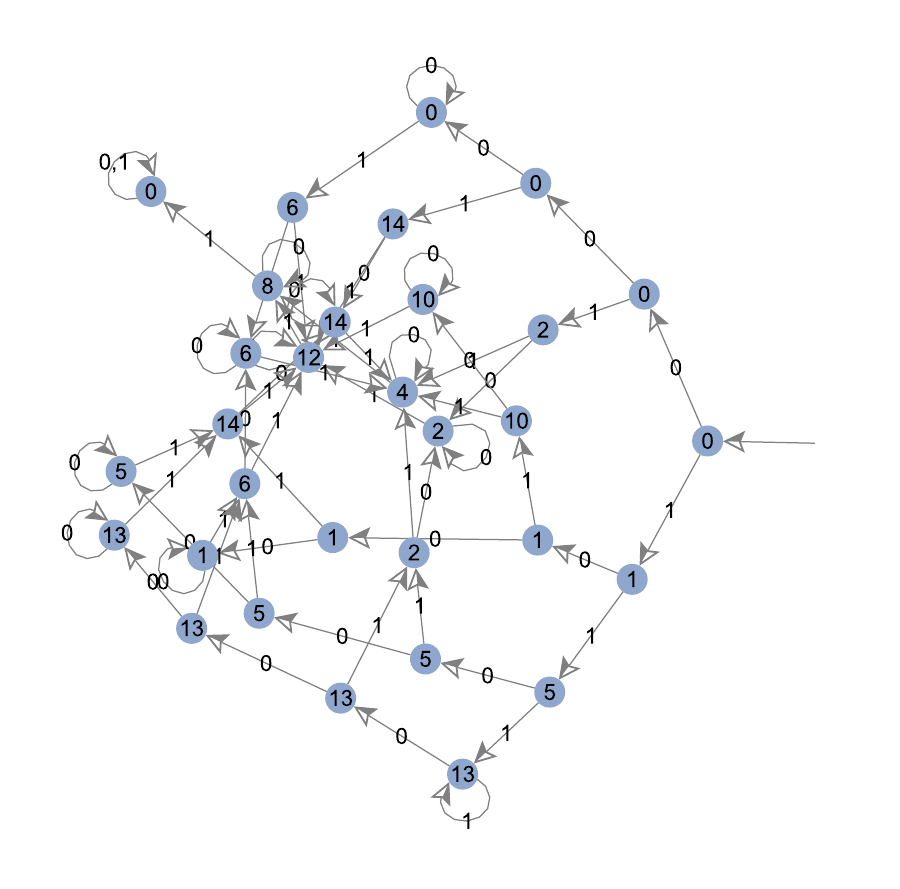}}
		\caption{Automata that compute Catalan numbers modulo $8$ (left) and $16$ (right).}
		\label{Catalan 8 and 16}
	\end{center}
\end{figure}

Automata for $C(n)$ modulo $8$ and $16$ appear in Figure~\ref{Catalan 8 and 16}.
In particular, we have the following, which was already explicit in the results of Liu and Yeh~\cite{Liu--Yeh}.

\begin{theorem}[Liu--Yeh]\label{Catalan mod 16}
For all $n \geq 0$, $C(n) \nequiv 9 \mod 16$.
\end{theorem}

That is, the residue class $9$ modulo $16$ is unattained by Catalan numbers, in addition to the classes $3, 7, 11, 15$ modulo $16$, which follow from $C(n) \nequiv 3 \mod 4$.
We omit automata for larger powers of $2$, but we record residues that are not attained.

\begin{theorem}
For all $n \geq 0$,
\begin{itemize}
\item
$C(n) \nequiv 17, 21, 26 \mod 32$,
\item
$C(n) \nequiv 10, 13, 33, 37 \mod 64$,
\item
$C(n) \nequiv 18, 54, 61, 65, 66, 69, 98, 106, 109 \mod 128$,
\item
$C(n) \nequiv 22, 34, 45, 62, 82, 86, 118, 129, 130, 133, 157, 170, 178, 253 \mod 256$,
\item
$C(n) \nequiv 6, 50, 93, 134, 142, 150, 162, 173, 210, 214, 220, 230, 242, 257 \mod 512$, \\
$C(n) \nequiv 258, 261, 270, 285, 294, 298, 348, 381, 382, 422, 446, 478, 502, 510 \mod 512$.
\end{itemize}
\end{theorem}

Only $180/512 \approx 35\%$ of the residues modulo $512$ are attained by some $C(n)$.
Higher powers of $2$ presumably also have new forbidden residues.
This suggests the following question.

\begin{question}
Does the fraction of residues modulo $2^\alpha$ that are attained by some Catalan number tend to $0$ as $\alpha$ gets large?
\end{question}

Lin~\cite{Lin} has shown that for $\alpha \geq 2$ there are exactly $\alpha - 1$ \emph{odd} residues attained modulo $2^\alpha$.
Hence there are $\alpha - 3$ essentially new forbidden odd residues modulo $2^\alpha$ for each $\alpha \geq 3$, and the fraction of odd residues attained does tend to $0$.
A result of Xin and Xu~\cite[Theorem~8]{Xin--Xu} gives some information regarding even residues.

There are also some residues modulo $2^\alpha$ that aren't missed completely but are only attained finitely many times.
This is also easy to determine from an automaton, by identifying the states that can be reached from the initial state in only finitely many ways.
For example, $C(n) \nequiv 1 \mod 8$ for all $n \geq 2$.
Similarly, $C(n) \nequiv 5, 10 \mod 16$ for all $n \geq 6$, and there are examples modulo higher powers of $2$ as well.

\begin{figure}
	\begin{center}
		\scalebox{.9}{\vcentergraphics{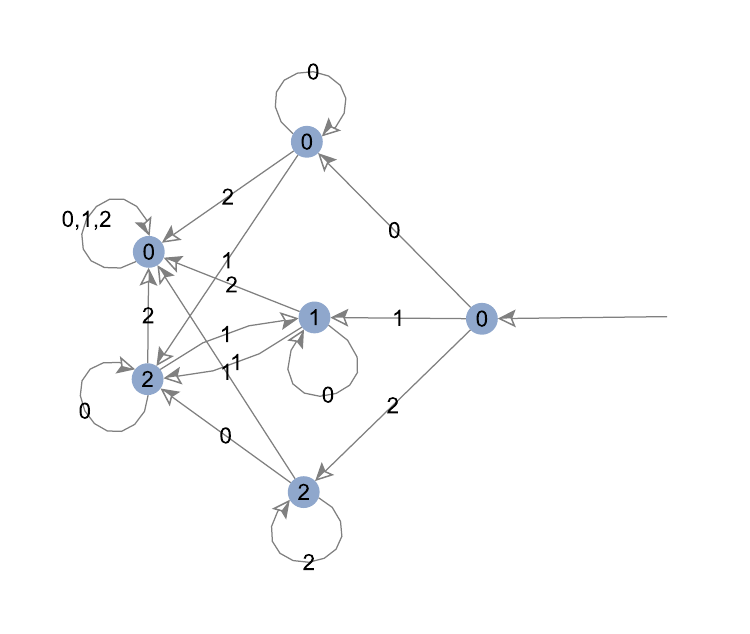}}
		\scalebox{.9}{\vcentergraphics{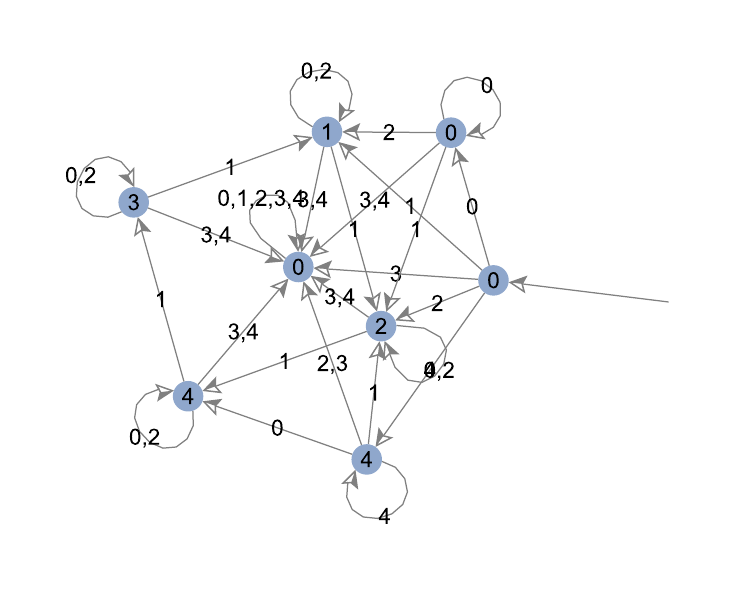}}
		\caption{Automata that compute Catalan numbers modulo $3$ (left) and $5$ (right).}
		\label{Catalan 3 and 5}
	\end{center}
\end{figure}

Automata for $C(n) \bmod p^\alpha$ can also be computed for powers of other primes.
For example, the automaton on the left in Figure~\ref{Catalan 3 and 5} computes $C(n) \bmod 3$; this result is a rephrasing of a theorem of Deutsch and Sagan~\cite[Theorem~5.2]{Deutsch--Sagan}.
More generally, we can produce an automaton for $C(n) \bmod 3^\alpha$ for any given $\alpha$, which corresponds to the congruences for $C(n)$ modulo $3^\alpha$ established by Krattenthaler and M\"{u}ller~\cite[Section~12]{Krattenthaler--Muller}.
However, we have not found any forbidden residues.

\begin{question}
Do there exist integers $\alpha \geq 1$ and $r$ such that $C(n) \nequiv r \mod 3^\alpha$ for all $n \geq 0$?
\end{question}

\subsection{Motzkin numbers}

Let $M(n)_{n \geq 0}$ be the sequence $1, 1, 2, 4, 9, 21, 51, 127, \dots$ of Motzkin numbers~\seq{A001006}.
The generating function $z = \sum_{n \geq 0} M(n) x^n$ for the Motzkin numbers satisfies
\[
	x^2 z^2 + (x - 1) z + 1 = 0.
\]

Deutsch, Sagan, and Amdeberhan~\cite[Conjecture~5.5]{Deutsch--Sagan} conjectured necessary and sufficient conditions for $M(n)$ to be divisible by $4$ and by $8$.
This conjecture was proved by Eu, Liu, and Yeh~\cite{Eu--Liu--Yeh}.
In particular, Motzkin numbers have a forbidden residue modulo $8$.

\begin{theorem}[Eu--Liu--Yeh]\label{Motzkin mod 8}
For all $n \geq 0$, $M(n) \nequiv 0 \mod 8$.
\end{theorem}

\begin{proof}
The series $0 + \sum_{n \geq 1} M(n) x^n$ satisfies
\[
	x^2 y^2 + (x + 1) (2 x - 1) y + x (x + 1) = 0.
\]
We apply the higher-degree variant of Algorithm~\ref{fast diagonal algorithm} described at the end of Section~\ref{easy_algorithm} to this series for $p^\alpha = 8$ to compute the automaton in Figure~\ref{Motzkin 8}.
It has $28$ states.
(Algorithm~\ref{fast diagonal algorithm} produces an automaton with $51$ states.)
Some states correspond to the output $0$, but the only incoming edges for these states are labeled $0$.
Since the binary representation of every integer $n \geq 1$ has most significant digit $1$, none of these states is the final state for an input $n \geq 1$.
On input $n = 0$, the automaton does output $0$, but this output is the constant term of the modified power series and not $M(0) \bmod 8$.
\end{proof}

\begin{figure}
	\begin{center}
		\includegraphics{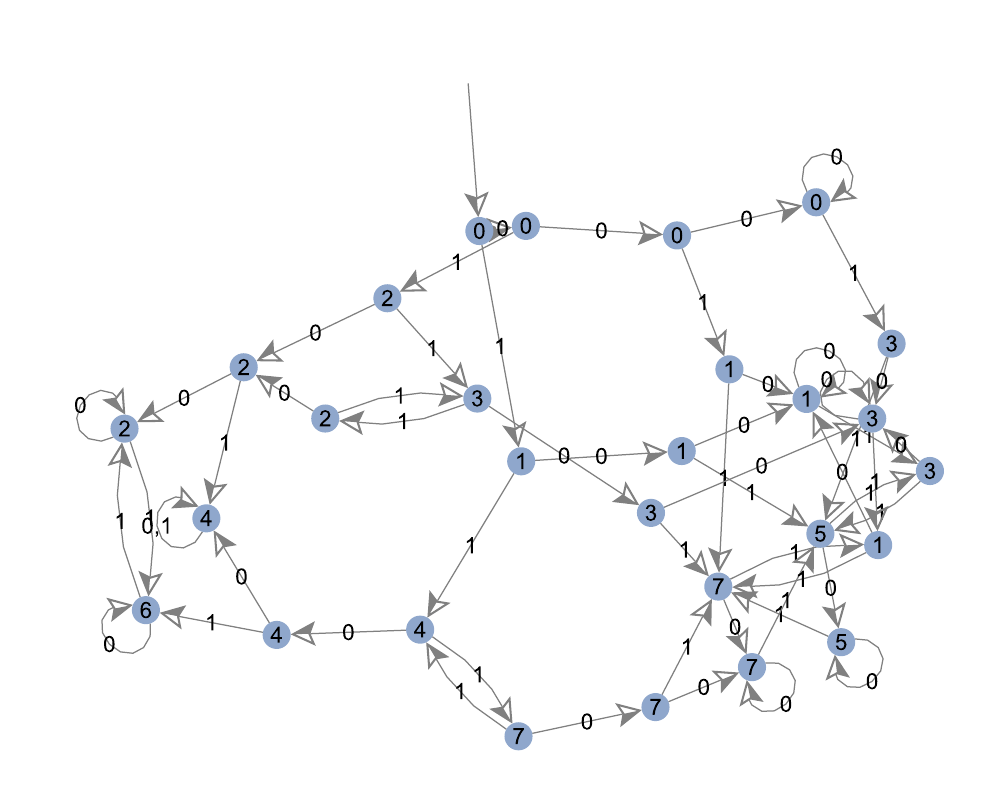}
		\caption{An automaton that computes Motzkin numbers modulo $8$.}
		\label{Motzkin 8}
	\end{center}
\end{figure}

Eu, Liu, and Yeh also gave an expression for $M(n) \bmod 8$ in the case that $M(n)$ is even.
The automaton in Figure~\ref{Motzkin 8} computes $M(n) \bmod 8$ in general.

Deutsch and Sagan~\cite[Corollary~4.10]{Deutsch--Sagan} determined the value of $M(n) \bmod 3$.
More generally, Krattenthaler and M\"{u}ller~\cite[Section~7]{Krattenthaler--Muller} showed how to produce congruences for $M(n) \bmod 3^\alpha$ for any given $\alpha$ in terms of power series.

Deutsch and Sagan~\cite[Theorem~5.4]{Deutsch--Sagan} also determined $M(n) \bmod 5$.
Modulo $5^2$, we can prove the following new theorem.

\begin{theorem}
For all $n \geq 0$, $M(n) \nequiv 0 \mod 5^2$.
\end{theorem}

The automaton computed by Algorithm~\ref{fast diagonal algorithm} for $M(n) \bmod 5^2$ has $144$ states.
We omit it here, but the explicit automaton is available from the web site of the first author.

The sequence of Motzkin numbers does not miss any residues modulo $3^2$, $7^2$, or $11^2$.
However, there is a forbidden residue modulo $13^2$.

\begin{theorem}
For all $n \geq 0$, $M(n) \nequiv 0 \mod 13^2$.
\end{theorem}

Algorithm~\ref{fast diagonal algorithm} produces an automaton for $M(n) \bmod 13^2$ with $2125$ states.
The computation took ten minutes on a $2.6$ GHz laptop with $8$ GB RAM.

The following conjecture is suggested by experimental evidence, but we have not been able to compute the automata for these moduli.

\begin{conjecture}
Let $p \in \{31, 37, 61\}$.
For all $n \geq 0$, $M(n) \nequiv 0 \mod p^2$.
\end{conjecture}

\begin{question}
Are there infinitely many primes $p$ such that $M(n) \nequiv 0 \mod p^2$ for all $n \geq 0$?
\end{question}

\subsection{Further applications}

Before considering other combinatorial sequences, we pause here to mention additional information that an automaton for $a_n \bmod p^\alpha$ provides access to.

First, one can often compute the distribution of the residues modulo $p^\alpha$ by computing the letter frequencies of $(a_n \bmod p^\alpha)_{n \geq 0}$.
A $p$-automatic sequence is the image, under a coding, of a fixed point of a $p$-uniform morphism $\varphi$.
If $\varphi$ is primitive, then the letter frequencies exist and are nonzero rational numbers~\cite[Theorems~8.4.5 and 8.4.7]{ash}.
Peter~\cite{Peter} gave necessary and sufficient conditions for the letter frequencies to exist in a general automatic sequence.
If the letter frequencies of the fixed point of $\varphi$ exist, then the vector of letter frequencies is an eigenvector of the incidence matrix of $\varphi$.

\begin{example}
The sequence $abccabab\cdots$ is a fixed point of the primitive morphism $\varphi(a) = ab$, $\varphi(b) = cc$, $\varphi(c) = ab$.
The image of this sequence under $a,b \mapsto 1$ and $c \mapsto 0$ is the sequence of Motzkin numbers modulo $2$~\seq{A039963}.
The incidence matrix of $\varphi$ is
\[
	A = \begin{bmatrix}
		1 & 0 & 1 \\
		1 & 0 & 1 \\
		0 & 2 & 0
	\end{bmatrix}.
\]
The Perron--Frobenius eigenvalue of $A$ is $2$, and $(1, 1, 1)/3$ is the corresponding eigenvector normalized so that the entries sum to $1$.
Hence the letters $a, b, c$ occur with equal frequency in the fixed point $abccabab\cdots$.
Therefore, in the sequence $(M(n) \bmod 2)_{n \geq 0}$ the letter $0$ occurs with frequency $1/3$ and the letter $1$ with frequency $2/3$.
\end{example}

Second, if the terms in $(a_n)_{n \geq 0}$ are not divisible by arbitrarily large powers of $p$ and $(a_n \bmod p^\alpha)_{n \geq 0}$ is $p$-automatic for some sufficiently large $\alpha$, then the sequence of $p$-adic valuations of $a_n$ is also $p$-automatic.
Let $\nu_p(n)$ be the exponent of the highest power of $p$ dividing $n$, with $\nu_p(0) = \infty$.

\begin{theorem}
If $(a_n)_{n \geq 0}$ is the sequence of coefficients of the diagonal of a rational expression $\frac{R(x_1, \dots, x_k)}{Q(x_1, \dots, x_k)}$ with $Q(0, \dots, 0) \nequiv 0 \mod p$ such that $\nu_p(a_n)_{n \geq 0}$ contains only finitely many distinct values, then $\nu_p(a_n)_{n \geq 0}$ is $p$-automatic.
\end{theorem}

\begin{proof}
Let $\alpha$ be an integer such that $\alpha > \nu_p(a_n)$ for all $n \geq 0$ satisfying $a_n \neq 0$.
Let $\mathcal{M}$ be an automaton computing $a_n \bmod p^\alpha$.
Apply the map $r \mapsto \nu_p(r)$ to the output label of each state in $\mathcal{M}$ to obtain a new automaton on the same graph.
The new automaton computes $\nu_p(a_n)$.
\end{proof}

\begin{example}
By Theorem~\ref{Motzkin mod 8}, the sequence $\nu_2(M(n))_{n \geq 0}$ is bounded.
It follows that $\nu_2(M(n))_{n \geq 0}$~\seq{A186034} is $2$-automatic.
Relabeling an automaton for $M(n) \bmod 4$ gives the automaton in Figure~\ref{Motzkin 2-adic valuation}.
One can even compute the letter frequencies of this sequence.
We already know that the limiting density of odd Motzkin numbers is $2/3$.
The limiting density of Motzkin numbers congruent to $2$ modulo $4$ is $1/6$, and the limiting density of Motzkin numbers congruent to $0$ modulo $4$ is $1/6$.
\end{example}

\begin{figure}
	\begin{center}
		\includegraphics{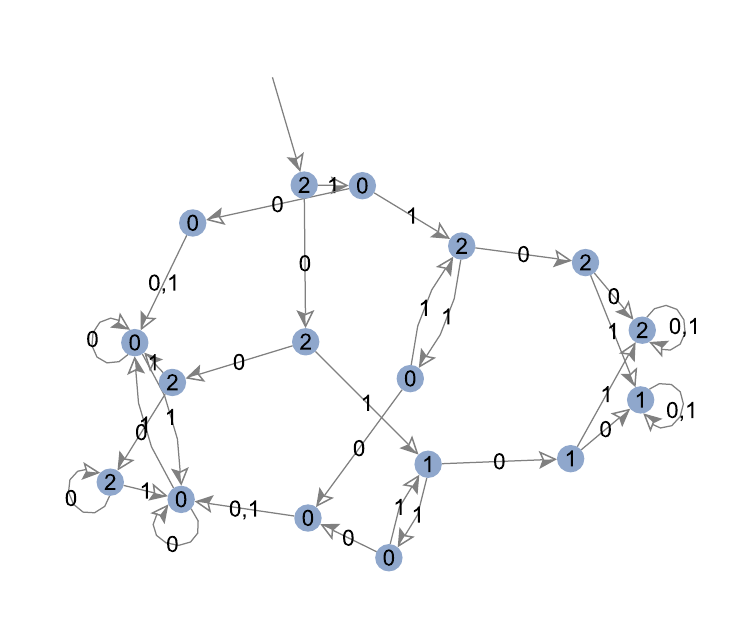}
		\caption{An automaton that computes $2$-adic valuations of Motzkin numbers.}
		\label{Motzkin 2-adic valuation}
	\end{center}
\end{figure}

If $\nu_p(a_n)$ takes on infinitely many distinct values, then $\nu_p(a_n)_{n \geq 0}$ cannot be $p$-automatic since its alphabet is not finite.
In this case, one would like to know whether $\nu_p(a_n)_{n \geq 0}$ is $p$-regular in the sense of Allouche and Shallit~\cite{Allouche--Shallit}.

\begin{question}
Is the sequence $\nu_3(M(n))_{n \geq 0}$ unbounded?
If so, is it $3$-regular?
\end{question}

\subsection{Well-known combinatorial sequences}

Of the many algebraic sequences that arise in combinatorics, we select just a few more for consideration.

The generating function of the sequence $R(n)_{n \geq 0} = 1, 0, 1, 1, 3, 6, 15, 36, \dots$ of Riordan numbers~\seq{A005043} satisfies
\[
	x (x + 1) z^2 - (x + 1) z + 1 = 0.
\]
Deutsch and Sagan~\cite[Corollaries~3.3 and 4.12]{Deutsch--Sagan} determined the value of $R(n)$ modulo $2$ and modulo $3$.
In particular, we have the following.

\begin{theorem}[Deutsch--Sagan]
For all $n \geq 0$, $R(n) \nequiv 2 \mod 3$.
\end{theorem}

Computing an automaton modulo $32$ shows the following.

\begin{theorem}
For all $n \geq 0$, $R(n) \nequiv 16 \mod 32$.
\end{theorem}

Let $P(n)_{n \geq 0}$ be the sequence $1, 1, 2, 5, 13, 35, 96, 267, \dots$ whose $n$th term is the number of directed animals of size $n$~\seq{A005773}.
Its generating function satisfies
\[
	(3 x - 1) z^2 - (3 x - 1) z + x = 0.
\]
Deutsch and Sagan~\cite[Corollaries~3.2 and 4.11]{Deutsch--Sagan} also determined the value of $P(n)$ modulo $2$ and $3$.
There are no forbidden residues modulo $2$ or $3$, but this sequence has the same forbidden residue modulo $32$ as the Riordan numbers.

\begin{theorem}
For all $n \geq 0$, $P(n) \nequiv 16 \mod 32$.
\end{theorem}

Finally, let $H(n)_{n \geq 0}$ be the sequence $1, 1, 3, 10, 36, 137, 543, 2219, \dots$ whose $n$th term is the number of restricted hexagonal polyominoes of size $n$~\seq{A002212}.
The generating function satisfies
\[
	x z^2 + (x - 1) z - x + 1 = 0.
\]
Again, Deutsch and Sagan~\cite[Corollary~3.4 and Theorem~5.3]{Deutsch--Sagan} determined the value of $H(n)$ modulo $2$ and $3$.
We can prove the following.

\begin{theorem}
For all $n \geq 0$, $H(n) \nequiv 0 \mod 8$.
\end{theorem}

\subsection{Sequences arising in pattern avoidance}

Pattern avoidance is a highly active area of study in combinatorics, and many combinatorial objects have been considered from this perspective.
Here we examine five sequences whose entries count trees or permutations avoiding a set of patterns.

If $a_n$ is the number of $(n+1)$-leaf binary trees avoiding a given finite set of contiguous patterns, then $(a_n)_{n \geq 0}$ is algebraic~\cite{Rowland_binary_trees}.
Two such sequences that exhibit forbidden residues modulo powers of $2$ are the following.

\begin{example}
Let $a_n$ be the number of $(n+1)$-leaf binary trees avoiding the following subtree.
\[
	\vcentergraphics{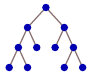}
\]
The sequence is $1, 1, 2, 5, 14, 41, 124, 384, \dots$~\seq{A159769}, and the generating function for this sequence satisfies
\[
	(x - 2) x^2 z^2 + (2 x^2 - 2 x + 1) z + x - 1 = 0.
\]
For all $n \geq 0$,
\begin{align*}
	a_n &\nequiv 3\phantom{0, 00, 00, 0}	\mod \phantom{0}4, \\
	a_n &\nequiv 25, 29\phantom{, 0, 00}	\mod 32, \\
	a_n &\nequiv 9, 13, 22, 37		\mod 64.
\end{align*}
\end{example}

\begin{example}
Let $a_n$ be the number of $(n+1)$-leaf binary trees avoiding the following subtree.
\[
	\vcentergraphics{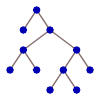}
\]
The sequence is $1, 1, 2, 5, 14, 41, 124, 385, \dots$~\seq{A159771}, and the generating function satisfies
\[
	2 x^2 z^2 - (3 x^2 - 2 x + 1) z + x^2 - x + 1 = 0.
\]
For all $n \geq 0$,
\begin{align*}
	a_n &\nequiv \phantom{0}3	\mod \phantom{0}4, \\
	a_n &\nequiv 13			\mod 16, \\
	a_n &\nequiv 21			\mod 32, \\
	a_n &\nequiv 37			\mod 64.
\end{align*}
\end{example}

Permutation patterns have received a huge amount of attention.
In general, sequences that count the permutations on $n$ elements avoiding a finite set of patterns are not algebraic.
However, some sets of patterns do yield algebraic sequences.

\begin{example}
Let $a_n$ be the number of permutations of length $n$ avoiding $3412$ and $2143$~\cite{Atkinson}.
The sequence is $1, 1, 2, 6, 22, 86, 340, 1340, \dots$~\seq{A029759}, and the generating function satisfies
\[
	(4 x - 1) (2 x - 1)^2 z^2 + (3 x - 1)^2 = 0.
\]
The coefficient of $x^0 z^1$ is $0$, so we might substitute $z = 1 + x y$;
however, the coefficient of $x^0 y^1$ is then $2$, which is not an obstruction for $p \neq 2$ but is an obstruction for $p = 2$.
Instead we use the fact that $a_n$ is even for all $n \geq 2$ and substitute $z = 
1 + x + 2 x y$.
Dividing the equation by $4 x$ then yields
\[
	x (4 x - 1) (2 x - 1)^2 y^2 + (x + 1) (4 x - 1) (2 x - 1)^2 y + x (4 x^3 + 3 x^2 - 4 x + 1) = 0,
\]
in which the coefficient of $x^0 y^1$ is $-1 \nequiv 0 \mod 2$.
Having divided the sequence by $2$, we need to multiply each output label by $2$ in any automaton modulo $2^\alpha$ we compute from this equation to recover an automaton for $a_n \bmod 2^{\alpha + 1}$.
For all $n \geq 0$,
\begin{align*}
	a_n &\nequiv 10, 14 &\mod \phantom{0}16, \\
	a_n &\nequiv 18 &\mod \phantom{0}32, \\
	a_n &\nequiv 34, 54 &\mod \phantom{0}64, \\
	a_n &\nequiv 44, 66, 102 &\mod 128, \\
	a_n &\nequiv 20, 130, 150, 166, 188, 204, 212, 214, 220, 236, 252 &\mod 256.
\end{align*}
\end{example}

\begin{example}
Let $a_n$ be the number of permutations of length $n$ avoiding $2143$ and $1324$~
\cite{Bona}.
The sequence is $1, 1, 2, 6, 22, 88, 366, 1552, \dots$~\seq{A032351}, and the generating function satisfies
\[
	(4 x^3 - 8 x^2 + 6 x - 1) z^2 + 2 (3 x^2 - 5 x + 1) z + x^2 + 4 x - 1 = 0.
\]
Again $a_n$ is even for $n \geq 2$.
The substitution $z = 1 + x + 2 x^2 + 2 x^2 y$ gives
\[
	(4 x^3 - 8 x^2 + 6 x - 1) y^2 + (8 x^3 - 12 x^2 + 8 x - 1) y + x (4 x^2 - 4 x + 3) = 0
\]
after dividing by $4 x^4$.
For all $n \geq 3$,
\begin{align*}
	a_n &\nequiv 2 &\mod \phantom{00}8, \\
	a_n &\nequiv 30 &\mod \phantom{0}32, \\
	a_n &\nequiv 14, 44, 54 &\mod \phantom{0}64, \\
	a_n &\nequiv 38, 46, 60, 76, 86 &\mod 128, \\
	a_n &\nequiv 92, 124, 140, 230, 238 &\mod 256, \\
	a_n &\nequiv 4, 12, 20, 110, 148, 150, 262, 278, 324, \\
		&\phantom{{} \nequiv {}} \quad 358, 372, 390, 412, 436, 454, 456, 476 &\mod 512.
\end{align*}
Note that the term $a_2 = 2$ was chopped in the substitution, so $a_n \nequiv 2 \mod 8$ holds only for $n \geq 3$.
\end{example}

\begin{example}\label{1342 and 2143}
Let $a_n$ be the number of permutations of length $n$ avoiding $1342$ and $2143$~\cite{Le}.
The sequence $(a_n)_{n \geq 0}$ is $1, 1, 2, 6, 22, 88, 368, 1584, \dots$~\seq{A109033}.
The generating function satisfies
\[
	2 x (x - 1) z^2 + z + x - 1 = 0.
\]
The automata for this sequence seem to have relatively few states, so they are fast to compute.
For all $n \geq 0$,
\begin{align*}
	a_n &\nequiv 3 &\mod \phantom{000}4, \\
	a_n &\nequiv 4, 5 &\mod \phantom{000}8, \\
	a_n &\nequiv 9, 10, 14 &\mod \phantom{00}16, \\
	a_n &\nequiv 8, 17, 18 &\mod \phantom{00}32, \\
	a_n &\nequiv 16, 33, 34, 38, 54 &\mod \phantom{00}64, \\
	a_n &\nequiv 24, 32, 65, 66, 70, 86, 120 &\mod \phantom{0}128, \\
	a_n &\nequiv 96, 129, 130, 134, 150, 176, 184, 240 &\mod \phantom{0}256, \\
	a_n &\nequiv 56, 112, 216, 224, 256, 257, 258, 262, 278, 304, 320, 448 &\mod \phantom{0}512, \\
	a_n &\nequiv 48, 192, 312, 513, 514, 518, 534, 600, 640, 856, 880, 984 &\mod 1024.
\end{align*}
\end{example}

\begin{question}
For the sequence $(a_n)_{n \geq 0}$ in Example~\ref{1342 and 2143}, is it true that
\[
	a_n \nequiv 2^{\alpha - 1} + 1, \, 2^{\alpha - 1} + 2, \, 2^{\alpha - 1} + 6, \, 2^{\alpha - 1} + 22 \mod 2^\alpha
\]
for all $\alpha \geq 6$ and $n \geq 0$?
\end{question}

\subsection{Ap\'ery numbers}\label{Apery}

The sequence of numbers $A(n) = \sum_{k=0}^n \binom{n}{k}^2 \binom{n+k}{k}^2$~\seq{A005259} arose in Ap\'ery's proof that $\zeta(3)$ is irrational.
Its generating function  $\sum_{n \geq 0} A(n) x^n$ is not algebraic, but it is the diagonal of the rational expression
\[
	\frac{1}{(1 - x_1 - x_2) (1 - x_3 - x_4) - x_1 x_2 x_3 x_4}
\]
in four variables~\cite{Straub}, so by Theorem~\ref{diagonal closure} the sequence $A(n) \bmod p^\alpha$ is $p$-automatic for every prime power $p^\alpha$.

Ap\'ery numbers modulo $p$ were studied by Gessel~\cite{Gessel}, who proved a Lucas-type theorem.
Namely, if $n = n_l \cdots n_1 n_0$ in base $p$, then
\[
	A(n) \equiv \prod_{i=0}^l A(n_i) \mod p.
\]
Therefore, one can easily verify that $A(n) \nequiv 0 \mod 2$ for all $n \geq 0$, since $A(0) \equiv A(1) \equiv 1 \mod 2$.
More generally, the primes $2, 3, 7, 13, 23, 29, 43, 47, \dots$~\seq{A133370} divide no Ap\'ery number.

\begin{question}
Are there infinitely many primes $p$ such that $A(n) \nequiv 0 \mod p$ for all $n \geq 0$?
\end{question}

One can generate an automaton for $A(n) \bmod p$ either using Gessel's theorem or Theorem~\ref{diagonal closure}, although Gessel's theorem accomplishes this more quickly since it specifies the structure directly.
We mention that $A(n) \bmod 7$ has a simple expression, which is evident from the automaton on the left in Figure~\ref{Apery 7 and 9}.

\begin{figure}
	\begin{center}
		\scalebox{.9}{\vcentergraphics{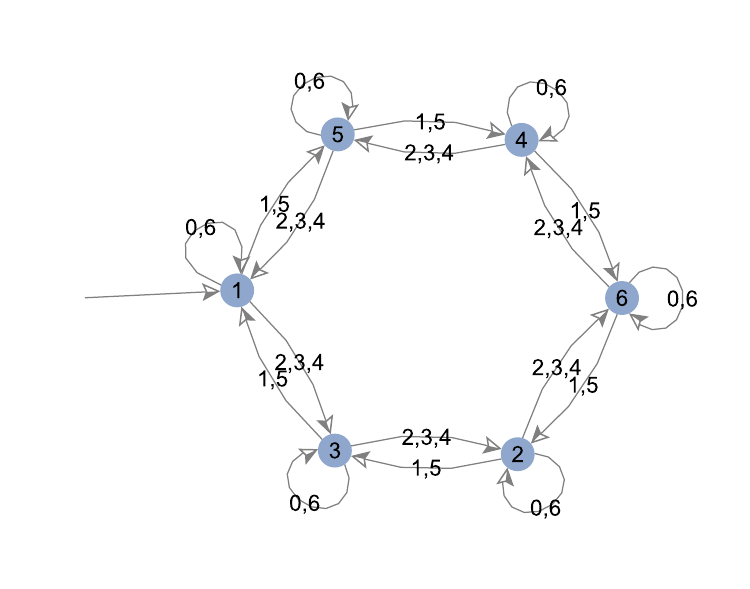}}
		\scalebox{.9}{\vcentergraphics{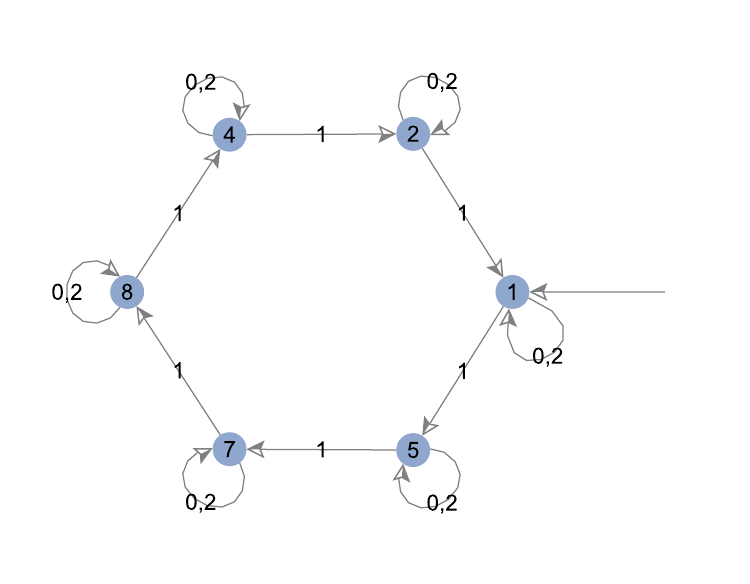}}
		\caption{Automata that compute Ap\'ery numbers modulo $7$ (left) and $9$ (right).}
		\label{Apery 7 and 9}
	\end{center}
\end{figure}

\begin{theorem}\label{Apery mod 7}
Let $e_d(n)$ be the number of occurrences of the digit $d$ in the standard base-$7$ representation of $n$.
For all $n \geq 0$,
\[
	A(n) \equiv 5^{e_1(n) + e_5(n) - e_2(n) - e_3(n) - e_4(n)} \mod 7.
\]
\end{theorem}

In addition, all edges labeled $0$ or $6$ in the automaton for $A(n) \bmod 7$ are loops, since $A(0) \equiv 1 \mod 7$ and $A(6) \equiv 1 \mod 7$.
That is, inserting or deleting any $0$s and $6$s in the base-$7$ representation of $n$ produces $n'$ such that $A(n) \equiv A(n') \mod 7$.

For $\alpha \geq 2$, the appropriate generalization of Algorithm~\ref{fast diagonal algorithm} allows us to resolve some conjectures by computing automata for $A(n) \bmod p^\alpha$.

Chowla, Cowles, and Cowles~\cite{Chowla--Cowles--Cowles} conjectured that
\[
	A(n) \bmod 8 =
	\begin{cases}
		1	& \text{if $n$ is even} \\
		5	& \text{if $n$ is odd.}
	\end{cases}
\]
This was proved by Gessel~\cite{Gessel}, who then asked whether $A(n)$ is periodic modulo $16$.
One computes the following automaton for $A(n) \bmod 16$.
\begin{center}
	\scalebox{.9}{\vcentergraphics{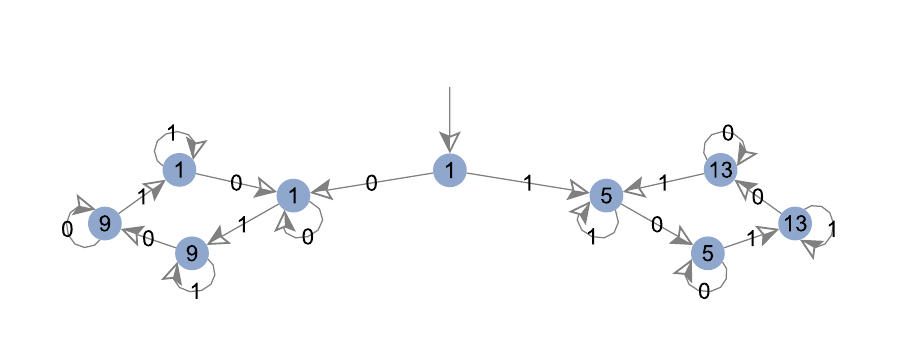}}
\end{center}
This automaton can be summarized as follows.
Let $\beta(n)$ be the number of blocks in the run-length encoding of the standard base-$2$ representation of $n$.
That is, for $n \geq 1$ write $n_l \cdots n_1 n_0 = 1^{\lambda_{\beta(n)}} \cdots d_2^{\lambda_2} d_1^{\lambda_1} d_0^{\lambda_0}$, where $\lambda_i \geq 1$, $d_i \in \{0, 1\}$, $d_i \neq d_{i+1}$, and $d_{\beta(n)} = 1$.
For $n = 0$ we have $\beta(0) = 0$.

\begin{theorem}\label{Apery mod 16}
For all $n \geq 0$, $A(n) \equiv 4 \beta(n) + 1 \mod 16$.
\end{theorem}

We use this structure to answer Gessel's question.
(Note that, more generally, it is decidable whether an automatic sequence is eventually periodic; see, for example, \cite{Allouche--Rampersad--Shallit}.)

\begin{theorem}
The sequence $(A(n) \bmod 16)_{n \geq 0}$ is not eventually periodic.
\end{theorem}

\begin{proof}
For each candidate period length $m \geq 1$, it suffices to exhibit arbitrarily large $n$ such that $A(n) \nequiv A(n + m) \mod 16$.
Write $m = m_l \cdots m_1 m_0$ in base $2$.
If $\beta(m) \nequiv 0 \mod 4$, let $n = 2^j$ for some $j \geq l + 2$; then $\beta(n) = 2$ and $\beta(n + m) = 2 + \beta(m) \nequiv 2 \mod 4$, so by Theorem~\ref{Apery mod 16} $A(n) \nequiv A(n + m) \mod 16$.
On the other hand, if $\beta(m) \equiv 0 \mod 4$, let $n = 2^j + 2^{l+1}$ for some $j \geq l + 3$; then $\beta(n) = 4$ and $\beta(n + m) = 2 + \beta(m) \equiv 2 \mod 4$, so again by Theorem~\ref{Apery mod 16} $A(n) \nequiv A(n + m) \mod 16$.
\end{proof}

Beukers~\cite{Beukers} conjectured that if there are $\alpha$ digits in the standard base-$5$ representation of $n$ that belong to the set $\{1, 3\}$ then $A(n) \equiv 0 \mod 5^\alpha$.
We can prove the statement for $\alpha = 2$.

\begin{theorem}
If two digits in the standard base-$5$ representation of $n$ belong to the set $\{1, 3\}$, then $A(n) \equiv 0 \mod 5^2$.
\end{theorem}

\begin{proof}
The automaton one computes for $A(n) \bmod 5^2$ has $29$ states $s_1, s_2, \dots, s_{29}$, indexed according to their positions in a breadth-first traversal of the automaton starting with the initial state $s_1$.
Only the states $s_2, s_4, s_6, s_8, s_{10}$ have an incoming edge labeled $1$ or $3$.
That is, after reading the first $1$ or $3$ in the base-$5$ representation of $n$, the automaton is in one of these five states.
The states that can be reached from these five states, in addition to themselves, are $s_5, s_7, s_{11}, s_{12}$.
It suffices to consider a second $1$ or $3$ read while in one of these nine states.
All eighteen edges labeled $1$ or $3$ leaving one of the nine states mentioned point to $s_6$.
All five edges leaving $s_6$ are loops, and the output corresponding to $s_6$ is $0$.
\end{proof}

The output corresponding to each of the nine states mentioned in the previous proof is a multiple of $5$.
Deleting these nine states produces the automaton in Figure~\ref{Apery 25} and the following theorem.

\begin{theorem}\label{Apery mod 25}
Let $e_2(n)$ be the number of $2$s in the standard base-$5$ representation of $n$.
If $n$ contains no $1$ or $3$ in base $5$, then $A(n) \equiv (-2)^{e_2(n)} \mod 25$.
\end{theorem}

\begin{figure}
	\begin{center}
		\scalebox{.8}{\vcentergraphics{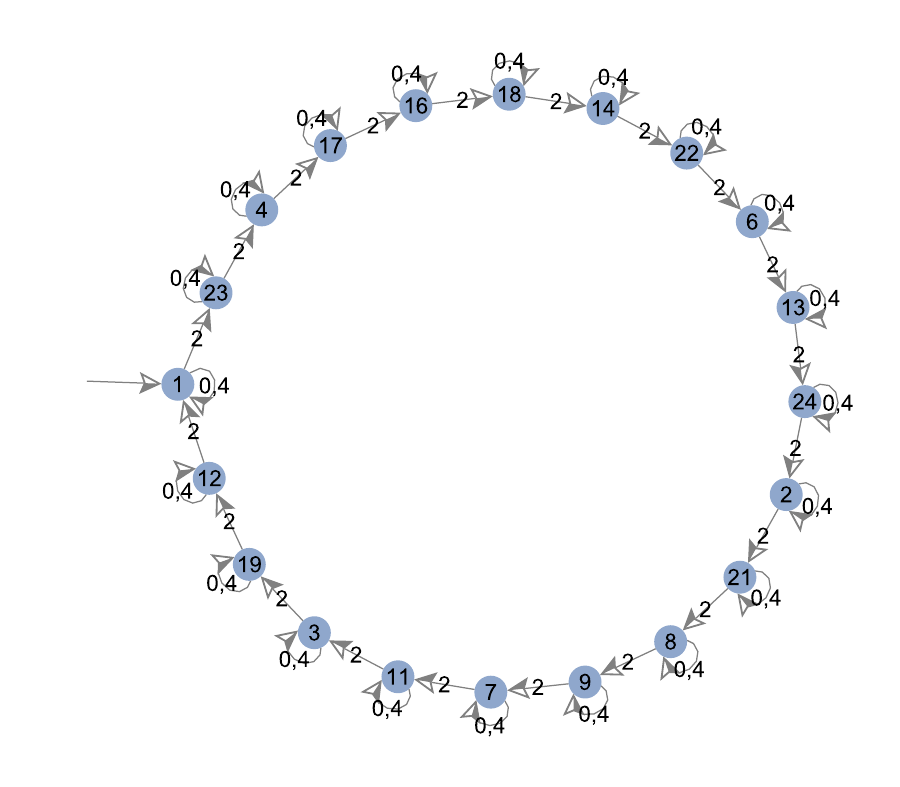}}
		\caption{A subautomaton computing certain Ap\'ery numbers modulo $25$.}
		\label{Apery 25}
	\end{center}
\end{figure}

Beukers~\cite{Beukers} also conjectured that if the standard base-$11$ representation of $n$ contains $\alpha$ occurrences of the digit $5$ then $A(n) \equiv 0 \mod 11^\alpha$.
These two conjectures were generalized by Deutsch and Sagan~\cite[Conjecture~5.13]{Deutsch--Sagan} to all primes and recently proved by Delaygue~\cite{Delaygue}.
In theory, we can compute an automaton for $A(n) \bmod p^\alpha$ for any prime power, and hence prove the conjecture for that prime power.
However, computing the automaton for $11^2$ or a larger prime power is computationally difficult in practice.

Krattenthaler and M\"uller~\cite[Conjecture~66]{Krattenthaler--Muller} conjectured the following (written in a slightly different form), which they were not able to prove with their method.

\begin{theorem}\label{Apery mod 9}
Let $e_1(n)$ be the number of $1$s in the standard base-$3$ representation of $n$.
For all $n \geq 0$, $A(n) \equiv 5^{e_1(n)} \mod 9$.
\end{theorem}

In fact this result was already proved by Gessel~\cite[Theorem~3(iii)]{Gessel}.
Computing the automaton on the right in Figure~\ref{Apery 7 and 9} gives a second proof.

Krattenthaler and M\"uller~\cite[Conjecture~65]{Krattenthaler--Muller} also gave a conjecture regarding $\sum_{k=0}^n \binom{n}{k}^2 \binom{n+k}{k}$~\seq{A005258}, which arose in Ap\'ery's proof of the irrationality of $\zeta(2)$.
The generating function of this sequence is the diagonal of
\[
	\frac{1}{(1-x_1) (1-x_2) (1-x_3) (1-x_4) - (1-x_1) x_1 x_2 x_3},
\]
so we prove this conjecture as well.
Krattenthaler and M\"uller~\cite{Krattenthaler--Muller2014} have since given another proof.

\begin{theorem}
The value of $\sum_{k=0}^n \binom{n}{k}^2 \binom{n+k}{k}$ modulo $9$ is given by the following base-$3$ automaton.
\begin{center}
	\scalebox{.8}{\vcentergraphics{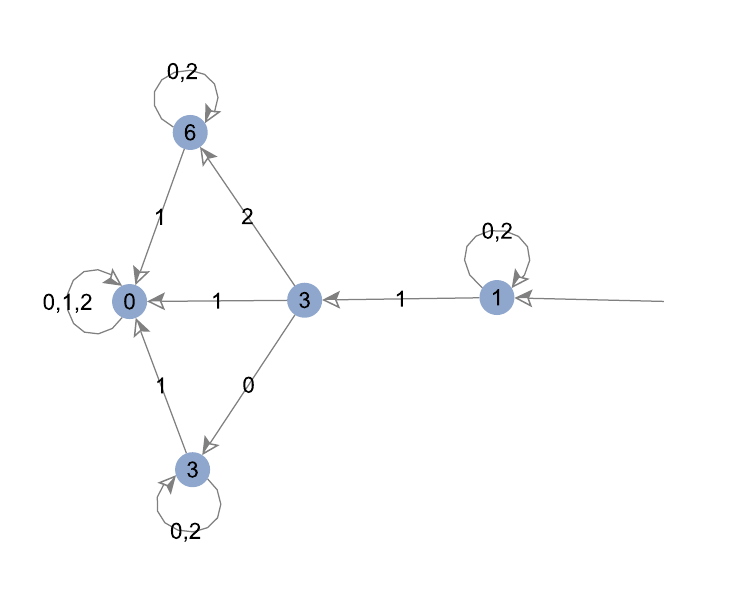}}
\end{center}
\end{theorem}

\section{Generating automata using Ore polynomials}\label{hard_algorithm}

\subsection{Theory}

In this section we take an alternate approach, based on \cite[Section~3]{Denef--Lipshitz}, to computing an automaton for $a_n \bmod p^\alpha$ for a given algebraic sequence $(a_n)_{n \geq 0}$ of $p$-adic integers.
Unlike the method discussed in Section~\ref{easy_algorithm}, this method works for every algebraic sequence modulo every prime power, with no restrictions on the coefficients of the polynomial.
However, rather than computing modulo $p^\alpha$ throughout, one first computes an automaton for the $p^0$ digit of $a_n$, then an automaton for the $p^1$ digit of $a_n$, etc., so that an automaton for $a_n \bmod p^\alpha$ is built up in $\alpha$ steps.
One might suspect that this method is therefore more computationally intensive, and this suspicion is substantiated by the presence of repeated computations involving Ore's lemma and resultants.

We first introduce projection and injection maps.
Identify $\F_p$ with $\{0, 1, \dots, p-1\}$, and define $\pi:\Z_p \rightarrow \F_p$ by $\pi(n_0 p^0 + n_1 p^1 + n_2 p^2 + \cdots) = n_0$, where $n_i \in \{0, 1, \dots, p-1\}$.
Define $\iota: \F_p \rightarrow \Z_p$ by $\iota(d) = d p^0 + 0 p^1 + 0 p^2 + \cdots$.
The maps $\pi$ and $\iota$ extend coefficient-wise to maps $\pi: \Z_p\llbracket x \rrbracket \rightarrow \F_p\llbracket x \rrbracket$ and $\iota: \F_p\llbracket x \rrbracket \rightarrow \Z_p\llbracket x \rrbracket$.

\begin{lemma}[Ore's lemma]\label{ore}
If $F(x) \in \F_p\llbracket x \rrbracket$ such that $P(x,F(x))= 0$ for some  $P(x,y) \in \F_p[x,y]$, then   there exists a polynomial $P^*(x,y) = g_0(x) y +g_1(x) y^p + \dots + g_m(x) y^{p^m} \in \F_p[x,y]$ such that $P^*(x,F(x))=0$ and $g_0(x) \neq 0$.
\end{lemma}

For a proof, see \cite[Lemma~12.2.3]{ash}.
We say that a polynomial $P^*(x,y) \in \F_p[x,y]$ satisfying the conclusions of Lemma~\ref{ore} is in {\em Ore form}.
The primary advantage of a polynomial in Ore form is that the partial derivative 
$P_y^* (x,y) \colonequal \frac{\partial}{\partial y} P^*(x,y) = g_0(x)$ is not the zero polynomial, so the following version of Hensel's lemma, whose proof can be found in~\cite[Remark 2.2]{Denef--Lipshitz},
 can be applied.

\begin{lemma}\label{Hensel}
Suppose that $F(x) \in \F_p\llbracket x \rrbracket$ and $P(x,y) \in \F_p[x, y]$ such that $P(x,F(x)) = 0$ and $P_y (x,y)|_{y=F(x)} \neq 0$. Then there exists $F^{(1)}(x) \in \Z_p\llbracket x \rrbracket$ such that $\iota(P)(x,F^{(1)}(x))= 0$ and $\pi(F^{(1)}(x)) = F(x)$.
\end{lemma}

Furthermore, Denef and Lipshitz~\cite[Lemma~3.4]{Denef--Lipshitz} proved that for any $\alpha$, the lifting of an algebraic series $F(x) \in \F_p\llbracket x \rrbracket$ into $\Z_p\llbracket x \rrbracket$ can be approximated modulo $p^\alpha$ by an algebraic series in $\Z_p\llbracket x \rrbracket$. The proof in Proposition~\ref{explicit_kernel}, below, of this result is essentially the same as theirs, although  we work with a different linear system. We begin with a lemma.

\begin{lemma}\label{Z_lemma}
Suppose $F(x) \in \F_p\llbracket x \rrbracket$ is algebraic.  Let $d = | \ker_p(F)|$, and let $f=\iota(F)$.  Then there exist polynomials $\ol b_0(x), \ldots ,\ol b_d(x)$  in $\Z[x]$ such that
\[
	\sum_{i=0}^{d} \ol b_i(x) f(x^{p^i}) = 0.
\]
\end{lemma}

\begin{proof}
Since $F(x)$ is algebraic,
then by Theorem~\ref{Christol} (Christol's theorem), we can compute $\ker_p(F) = \{F_1, \ldots, F_d \}$, where $F_1 = F$; this also implies  that 
$\ker_p(f) = \{f_1, \ldots, f_d \}$, where $f_i \colonequal \iota(F_i)$, so that $f_1=f$.
   
   We now retrace some of the steps in the proof of Theorem~\ref{Christol} to obtain a certain linear relationship. Writing each kernel element $f_i(x)= \sum_{j=0}^{p-1} x^j \lam_j (f_i)(x^p)$, and noting that each $\lam_j(f_i)$ is again an element in $\{f_1, \ldots, f_d \}$, we conclude that for $1\leq i \leq d$,  $f_i(x)$ is in the $\Z[x]$-linear span  of $\{ f_1(x^p), \ldots, f_d (x^p)\}$.
Replacing $x$ with $x^{p^l}$, we can conclude that for each $l$, and for each $i$,  $f_i(x^{p^{l}})$ is in the $\Z[x]$-linear span  of $\{ f_1(x^{p^{l+1}}), \ldots, f_d (x^{p^{l+1}})\}$. This implies that  for $0\leq j \leq d$, each $f_1(x^{p^j})$ lives in $M$, the $\Z[x]$-submodule of $\Z\llbracket x \rrbracket$
  spanned by $\{ f_i(x^{p^{d}}): 1\leq i \leq d \}$.   For each $j\in \{0, \ldots, d\}$, let the polynomials  $a_{1,j}(x), \ldots, a_{d,j}(x)\in \Z[x]$ be the coefficients in
   $f_1(x^{p^j})= \sum_{i=1}^d a_{i,j} (x)\, f_i(x^{p^{d}})$. We solve this linear system, describing $f_1(x) , f_1(x^{p})  \ldots , f_1(x^{p^d})$ in terms of $f_1(x^{p^d}), \ldots,    f_d(x^{p^d})$,
  over $\Z ( x )$ to get coefficients $\{\wt b_0(x), \ldots, \wt b_{d}(x)\}$ in $\Z( x ) $ such that
   $\sum_{i=0}^{d} \wt b_i(x) f_1(x^{p^{i}})     = 0$. We then multiply by a common denominator to get coefficients $\{\ol {b_0}(x), \ldots, \ol {b_{d}}(x)\}$ in $\Z[x] $ such that
   $\sum_{i=0}^{d} \ol{b_i}(x)  f_1(x^{p^{i}})  = 0$. 
\end{proof}

\begin{proposition}\label{explicit_kernel}
Suppose that $F(x) \in \F_p\llbracket x \rrbracket$ and $P(x,y) \in \F_p[x,y]$ such that $P(x, F(x)) = 0$, and let $\alpha \geq 1$.
Then there exists an algebraic $F^{(\alpha)}(x) \,\in \Z_p\llbracket x \rrbracket$ such that $\iota(F(x)) \equiv F^{(\alpha)} (x) \mod p^\alpha$.
Furthermore, a polynomial $P^{(\alpha)}(x,y) \in \Z[x,y]$ such that $P^{(\alpha)}(x,F^{(\alpha)}(x))= 0$ can be explicitly computed.
\end{proposition}

\begin{proof}

First we let $\alpha=1$. Given $P$, from Lemma~\ref{ore} we obtain a $P^*(x,y) \in \F_p[x,y]$ in Ore form such that $F(x)$ is a root of $P^*$.
Let $P^{(1)} = \iota(P^*)$. 
Now Lemma~\ref{Hensel} implies that there is an algebraic $F^{(1)}(x) \in \Z_p\llbracket x \rrbracket$ that is a root of $P^{(1)}$,  which agrees with $\iota(F(x))$ modulo $p$.

Inductively, suppose that we have computed a polynomial $P^{(\alpha-1)}(x,y) \in \Z[x,y]$ such that $P^{(\alpha -1)}(x, F^{(\alpha-1)}(x)) = 0$ for some $F^{(\alpha -1)}$ with $\iota(F(x)) \equiv F^{(\alpha-1)} (x) \mod p^{\alpha-1}$.
To determine $F^{(\alpha )}$, define $\delta(x) \in \Z_p\llbracket x \rrbracket$ by
\begin{equation} \label{delta_def}\delta (x)\colonequal \frac{F^{(\alpha -1)} (x) -  \iota(F(x))}{p^{\alpha -1}}.\end{equation}
Our aim is to show that $\delta(x) \equiv \Delta(x) \mod p$  for some algebraic $\Delta (x) \in \Z_p\llbracket x \rrbracket$, and to compute a polynomial $Q(x,y)$ 
which has $\Delta(x)$ as a root.
 Equation \eqref{delta_def} will then imply that 
\[ \iota(F(x))\equiv F^{(\alpha -1)}(x) - p^{\alpha -1}\Delta (x) \mod p^{\alpha },\] 
so we may take $F^{(\alpha )}(x) \colonequal F^{(\alpha -1)}(x) - p^{\alpha-1}\Delta(x)$; using  $P^{(\alpha -1)}$ and $Q$, we can then compute 
a polynomial $P^{(\alpha )}(x,y) \in \Z[x,y]$ such that $P^{(\alpha )}(x,  F^{(\alpha)}(x) )= 0$.

 Let $\mathcal A = \Z/(p^{\alpha}\Z)$. Then  $\mathcal{A}[x]$  is a quotient of the integral domain $\Z[x]$. Since $F(x)$ is algebraic over $\F_p(x)$, then using the notation of Lemma~\ref{Z_lemma},  we have polynomials $\ol{b_i}(x) \in \Z[x]$ such that  
 $\sum_{i=0}^{d} \ol{b_i}(x)  f(x^{p^{i}})  = 0$.   Project $\overline b_i(x)$ to 
$b_i(x)  \in \mathcal{A}[x]$   (after dividing by sufficiently many powers of $p$ if necessary, to get a nonzero linear combination modulo $p^\alpha$), 
 so that we explicitly obtain a linear relationship

\begin{equation} \label{L}
\sum_{i=0}^{d}b_i(x) f(x^{p^i }) 
\equiv0 \mod p^{\alpha }.\end{equation}
Noting that $f = \iota(F)$, and
applying Relationship \eqref{L} to Equation \eqref{delta_def},  we obtain
\[  \sum_{i=0}^{d}b_i(x) \left(   F^{(\alpha -1)}(x^{p^{i}}) - p^{\alpha-1} \delta(x^{p^{i}})   \right)  
   \equiv 0 \mod p^{\alpha }, 
\] so that 
\[   \sum_{i=0}^{d}b_i(x)    F^{(\alpha -1)}(x^{p^{i}})   \equiv 
 p^{\alpha -1} \sum_{i=0}^{d}b_i(x)   \delta(x^{p^{i}}) 
   \mod p^{\alpha }.
\]
Since $\delta(x^{p^i}) \equiv \delta(x)^{p^i} \mod p$, we have
\[ 
\sum_{i=0}^{d}b_i(x)   F^{(\alpha -1)}(x^{p^{i}})  \equiv 
 p^{\alpha -1} \sum_{i=0}^{d}b_i(x)   \delta(x)^{p^i} 
   \mod p^{\alpha }.
\]

Since $P^{(\alpha -1)}(x, F^{(\alpha -1)}(x))=0$, we can find a $Q^{**}(x,y) \in \Z[x,y]$ such that
\[
	Q^{**}\left(x, \frac{1}{p^{\alpha - 1}} \sum_{i=0}^d b_i(x)  F^{(\alpha -1)}(x^{p^{i}})\right)=0.
\]
Thus we have that
\[Q^{**}\left(x,             \sum_{i=0}^{d}b_i(x)   \delta(x)^{p^i} 
\right) \equiv 0 \mod p.\]
Thus $\pi(\delta(x)) \in \F_p\llbracket x \rrbracket$ is a root of   $\pi (Q^*(x,y))$ for some $Q^*(x,y)$. We put $\pi(Q^*)$  in Ore form using Lemma~\ref{ore} to get a polynomial
$Q(x,y) \in \F_p[x,y]$ such that $Q(x, \pi(\delta(x)))=0$. Using Hensel's lemma, we lift $\pi(\delta(x))$ to $\Delta(x) \in \Z_p\llbracket x \rrbracket$, a root of $\iota(Q)$. We have $\Delta(x) \equiv \delta(x) \mod p$, and the proof of the proposition is now complete.
\end{proof}

In this proof we have used the fact that algebraic power series form a ring.
Given polynomials satisfied by two algebraic power series, polynomials satisfied by their sum and their product can be computed using resultants.
Such polynomials may not be irreducible and so may have unnecessarily high degree.
We employ a standard symbolic--numeric technique to keep computations involving multiple resultants manageable.
Given two power series we would like to add or multiply, compute the first $N$ terms of each series for some $N$ (for example, $N = 64$).
Then add or multiply these truncated series.
Use a resultant to compute a polynomial for which the sum or product of the full series is a root.
Then factor this polynomial, and evaluate each factor, up to $O(x^N)$, at the truncated sum or product.
If there is only one factor that evaluates to $0 + O(x^N)$, then use this factor instead of the full polynomial.

For $a_n \in \Z_p$, let $a_n(i)$ be the $i$th base-$p$ digit of $a_n$, so that $a_n = a_n(0) p^0 + a_n(1) p^1 + a_n(2) p^2 + \cdots$.
Using Proposition~\ref{explicit_kernel}, Denef and Lipshitz showed that if $f(x) \in \Z_p\llbracket x \rrbracket$ is algebraic, then projecting the coefficients of $f(x)$ to their $i$th digits yields an algebraic series over $\F_p(x)$.

\begin{proposition}[Denef--Lipshitz {\cite[Proposition 3.5]{Denef--Lipshitz}}] \label{projections_are_algebraic}
Suppose that $f(x) = \sum_{n \geq 0} a_n x^n \in \Z_p\llbracket x \rrbracket$ is algebraic, where we are given $P(x,y) \in \Z[x,y]$ such that $P(x,f(x))=0$. Then for each $i$, $f_i(x) \colonequal \sum_{n \geq 0} a_n(i) x^n \in \F_p\llbracket x \rrbracket$ is algebraic, and a polynomial $R_i(x,y) \in \F_p[x,y]$ can be computed such that $R_i(x,f_i(x))=0$.
\end{proposition}

\begin{proof}
The constructive nature of this proposition is clear in the proof of  Proposition 3.5 in \cite{Denef--Lipshitz}; we reiterate their inductive proof. Note that 
if $f$ is a root of $P$, then $f_0$ is a root of $\pi ( P )$. 
(The only thing to note is that we need $\pi ( P ) \neq 0$. If the coefficients of $P$ are all divisible by $p^k$, we divide $P$ by $p^k$. In this way we can assume that the projection of $P$ is nonzero modulo $p$.)

 If $i\geq1$,
\[ f_i = \pi\left(\frac{f -  \sum_{j=0}^{i-1} p^j\iota(f_j)}{p^i}\right). \]
Assuming that we have shown that $f_0, f_1, \ldots, f_{i-1}$ are algebraic, then
by Proposition~\ref{explicit_kernel},  for each $j \in \{0, \ldots, i-1 \}$ there exists $F_j^{(i+1-j)} \in \Z_p\llbracket x \rrbracket $ such that 
$F_j^{(i+1-j)}\equiv \iota(f_j) \mod p^{i+1-j}$. This means that 
\[ f_i = \pi\left(\frac{f -  \sum_{j=0}^{i-1} p^j  F_j^{(i+1-j)}}{p^i}\right), \]
i.e.\ $f_i$ is algebraic. Furthermore 
Proposition~\ref{explicit_kernel} tells us that for $j \in \{0, \ldots, i-1 \}$  we can find polynomials 
$P_j^{(i+1-j)} $ such that $P_j^{(i+1-j)} (x, F_j^{(i+1-j)})=0$; we can use these polynomials to compute a polynomial $P_i$ such that $P_i(x,f_i(x))=0$.
\end{proof}

Both Propositions~\ref{explicit_kernel} and 
\ref{projections_are_algebraic} generalize to multivariate series; we have confined ourselves to univariate series for notational simplicity.  

Suppose that $f  =  \sum_{n \geq 0} a_n x^n \in \Z_p\llbracket x \rrbracket$ is algebraic. Proposition~\ref{projections_are_algebraic} tells us that for each $i$,  $f_i =  \sum_{n \geq 0} a_n(i) x^n \in \F_p\llbracket x \rrbracket$ is the root of a computable polynomial. The constructive proofs of Theorems~\ref{Christol} and \ref{Eilenberg} allow us to explicitly describe the $p$-kernel of $f_i$.
By computing an automaton for the coefficients of each series $f_0, \ldots, f_{\alpha-1}$ and taking the direct product of these automata, we obtain an automaton which computes the $n$th coefficient of each series simultaneously.
Therefore we have the following theorem.

\begin{theorem}[Denef--Lipshitz {\cite[Theorem 3.1(i)]{Denef--Lipshitz}}] \label{algebraic_implies_automatic_projections}
Suppose that the power series $f(x_1,\ldots, x_k) \in \Z_p\llbracket x_1, \ldots, x_k \rrbracket$ is algebraic over $\Z_p(x_1, \ldots, x_k)$, where we are given a polynomial $P(x_1, \ldots, x_k,y) \in \Z[x_1, \ldots, x_k,y]$ such that $P(x_1,\ldots,x_k,f(x))=0$. Then for each $\alpha$, the coefficient sequence of $f \bmod p^\alpha$ is $p$-automatic.
\end{theorem}

\begin{remark} \label{second estimate} If $P(x,f(x)) =0$ and $P^*(x,y) = \sum_{i=0}^{K}g_i(x) y^{p^{i}}$ is the Ore form of $P(x,y)$,
Adamczewski and Bell~\cite[Lemma 8.1]{Adamczewski--Bell--2}
gave bounds on the $x$-degree of $P^*(x,y)$. Namely, if $H = \deg_x P(x,y)$ and $K= \deg_y  P(x,y)$, then the degree of each $g_i(x)$ is at most $HK p^{K}$. Now an automaton which computes $a_n \bmod p$ can be obtained from the proof of Theorem~\ref{Christol}, for example as given in \cite[Theorem~12.2.5]{ash}. The vertices of the automaton can be injected into a space whose elements can be described using $K+1$ polynomials in $\F_p[x]$, each of whose degree is at most $(p^K -1)HK p^K$. Thus the vertex set of the automaton has size at most 
$p^{(K+1)((p^{K}-1)HKp^K +1)}$. We contrast this with the size of the automaton generated by Proposition~\ref{Furstenberg} and Theorem~\ref{diagonal closure}.
If $f(0)=0$ and $P_y(0,0) \nequiv 0 \mod p$, then $f(x) = \mathcal D \left (\frac{y P_y(xy,y)}{P(xy,y)/y}\right)$; according to Remark~\ref{first estimate}, there is an automaton computing $a_n \bmod p$ with at most $p^{(L+1)^2}$ states, where
\begin{align*}
	L &= \max\{\deg y P_y(xy,y), \deg P(xy,y)/y\} \\
	&\leq \deg_x P(x,y) + \deg_y P(x,y) \\
	&= H+K.
\end{align*}
For large $p$ these bounds suggest that the automaton obtained using Theorem~\ref{diagonal closure} is far smaller than the one obtained using the methods of this section.
This emphasizes the computational burden of using polynomials in Ore form, although these bounds may not be representative of typical automata. If $\alpha>1$, then the repeated use of Ore's lemma (at least $\alpha$ times) and the resultant (estimates of the cost of which are given in~\cite[Lemma 4.1]{Adamczewski--Bell}) make the bound on the size of an automaton for $a_n \bmod p^{\alpha}$ even larger.
\end{remark}

Given an algebraic series $f(x) \in \Z_p\llbracket x \rrbracket$, Algorithm~\ref{slow lifting algorithm} is the algorithm suggested by Propositions~\ref{explicit_kernel} and \ref{projections_are_algebraic} to compute an automaton for the coefficients of $f(x)$ modulo $p^\alpha$ (where of course we carry around each series by a polynomial).

\begin{algorithm}\label{slow lifting algorithm}
\SetArgSty{}
\DontPrintSemicolon
\KwIn{$(P(x, y), p, \alpha) \in \mathcal{\Z}[x, y] \times \mathbb{P} \times \Z_{\geq 1}$ where $P(x, f(x)) = 0$}

$f_0(x) \leftarrow \pi(f(x))$\;

Compute an automaton $\mathcal M_0$ for the coefficients of $f_0(x)$\;

\For{$1 \leq i \leq \alpha - 1$}{

	Use $\mathcal M_{i-1}$ to compute $\ol {b_0}(x), \ldots, \ol {b_{d}}(x)$ as in Lemma~\ref{Z_lemma}\;

	\For{$2 \leq j \leq \alpha - i + 1$}{

		Compute $F_{i-1}^{(j)}(x) \equiv \iota(f_{i-1}(x)) \mod p^j$ as in Proposition~\ref{explicit_kernel}\;

	}

	$f_i(x) \leftarrow \pi((f(x) - \sum_{j=0}^{i-1} p^j F_j^{(i+1-j)}(x))/p^i)$\;

	Compute an automaton $\mathcal M_i$ for the coefficients of $f_i(x)$\;

}

Compute the direct product $\mathcal M$ of $\mathcal M_0, \dots, \mathcal M_{\alpha - 1}$\;

\Return $\mathcal M$\;
\caption{Outline for computing an automaton for the coefficients of an algebraic sequence modulo $p^\alpha$.}
\end{algorithm}

However, this algorithm includes some unnecessary computations.
For example, we see
\[
	f_1(x)
	= \pi\left(\frac{f(x) - p^0 F_0^{(2)}(x)}{p}\right)
	= \pi\left(\frac{f(x) - \left(F_0^{(1)}(x) - p \Delta(x)\right)}{p}\right)
	= \pi(\Delta(x))
\]
for some algebraic $\Delta(x)$ that we compute in Proposition~\ref{explicit_kernel}.
For $\alpha = 2$, therefore, we do not in fact need to compute a polynomial for $F_0^{(2)}(x)$.
More generally, the computation can be carried out in terms of the various series $\Delta(x)$ rather than the lifted components $F_i^{(j)}(x)$.
Let $\Delta_i^{(k)}(x)$ be the series $\Delta(x)$ used to compute $F_i^{(k)}(x) \colonequal F_i^{(k - 1)}(x) - p^{k - 1} \Delta(x)$ in Proposition~\ref{explicit_kernel}.
Then observe that in the proof of Proposition~\ref{explicit_kernel}, for $\alpha = 1$ we can take $P^{(1)}=Q$ and $ F^{(1)}=f$, where $f(x) \in \Z_p\llbracket x \rrbracket$ is a root of $Q$.
Therefore for $1 \leq j \leq \alpha$ we have
\[
	F_0^{(j)} = f - \sum_{k=2}^{j} p^{k-1} \Delta_0^{(k)}.
\]
For $2 \leq i \leq \alpha - 1$ and $1 \leq j \leq \alpha - i + 1$ we have
\[
	F_{i-1}^{(j)} = \sum_{k=2}^i \Delta_{i-k}^{(k)} - \sum_{k=2}^{j} p^{k-1} \Delta_{i-1}^{(k)}.
\]
When we write $f_i$ in terms of $\Delta_j^{(k)}$, most of the terms cancel and we are left with
\[
	f_i
	= \pi\left(\frac{f -  \sum_{j=0}^{i-1} p^j  F_j^{(i+1-j)}}{p^i}\right)
	= \pi\left( \sum_{j=2}^{i+1} \Delta_{i+1-j}^{(j)} \right).
\]
The resulting algorithm is Algorithm~\ref{lifting algorithm}.

\begin{algorithm}\label{lifting algorithm}
\SetArgSty{}
\DontPrintSemicolon
\KwIn{$(P(x, y), p, \alpha) \in \mathcal{\Z}[x, y] \times \mathbb{P} \times \Z_{\geq 1}$ where $P(x, f(x)) = 0$}

$f_0(x) \leftarrow \pi(f(x))$\;

Compute an automaton $\mathcal M_0$ for the coefficients of $f_0(x)$\;

\For{$1 \leq i \leq \alpha - 1$}{

	Use $\mathcal M_{i-1}$ to compute $\ol {b_0}(x), \ldots, \ol {b_{d}}(x)$ as in Lemma~\ref{Z_lemma}\;

	\For{$2 \leq j \leq \alpha - i + 1$}{

		Compute $\Delta_{i-1}^{(j)}(x)$ as in Proposition~\ref{explicit_kernel}\;

	}

	$f_i(x) \leftarrow \pi(\sum_{j=2}^{i+1} \Delta_{i+1-j}^{(j)}(x))$\;

	Compute an automaton $\mathcal M_i$ for the coefficients of $f_i(x)$\;

}

Compute the direct product $\mathcal M$ of $\mathcal M_0, \dots, \mathcal M_{\alpha - 1}$\;

\Return $\mathcal M$\;
\caption{Outline for computing an automaton for the coefficients of an algebraic sequence modulo $p^\alpha$, using fewer operations than Algorithm~\ref{slow lifting algorithm}.}
\end{algorithm}

\subsection{Central trinomial coefficients}\label{Central trinomial coefficients}

As an example, let $(T_n)_{n \geq 0}$ be the sequence $1, 1, 3, 7, 19, 51, 141, 393, \dots$ of central trinomial coefficients~\seq{A002426}.
The generating function $z = \sum_{n \geq 0} T_n x^n$ satisfies
\[
	(x + 1) (3 x - 1) z^2 + 1 = 0.
\]
Since $T_0 \neq 0$, we might substitute $z = y + 1$ in an attempt to obtain a polynomial satisfying the conditions of Section~\ref{easy_algorithm}.
The polynomial $P(x, y)$ one obtains has $\frac{\partial P}{\partial y}(0, 0) = 2 \equiv 0 \mod 2$, so we cannot use it to compute an automaton for $T_n \bmod 2^\alpha$ using Proposition~\ref{Furstenberg} and Algorithm~\ref{fast diagonal algorithm}.
Truncating additional terms of the power series does not fix the problem.
Indeed, we have not been able to apply the method of Section~\ref{easy_algorithm} to this sequence.
Therefore we carry out Algorithm~\ref{lifting algorithm} to compute an automaton for $T_n \bmod 4$.

Let $P(x, y) = (x + 1) (3 x - 1) y^2 + 1$.
Projecting modulo $2$ shows that the series $f_0(x) = \sum_{n \geq 0} T_n(0) x^n \in \F_2\llbracket x \rrbracket$ satisfies $(x + 1)^2 y^2 + 1 = 0$.
An Ore form for $P(x, y)$ modulo $2$ is
\[
	(x + 1) y^2 + y.
\]
From this one computes an automaton $\mathcal M_0$ for $T_n \bmod 2$ using Theorem~\ref{Christol}.
This automaton is as follows, showing that $T_n \equiv 1 \mod 2$ for all $n \geq 0$.
\begin{center}
	\scalebox{.7}{\includegraphics{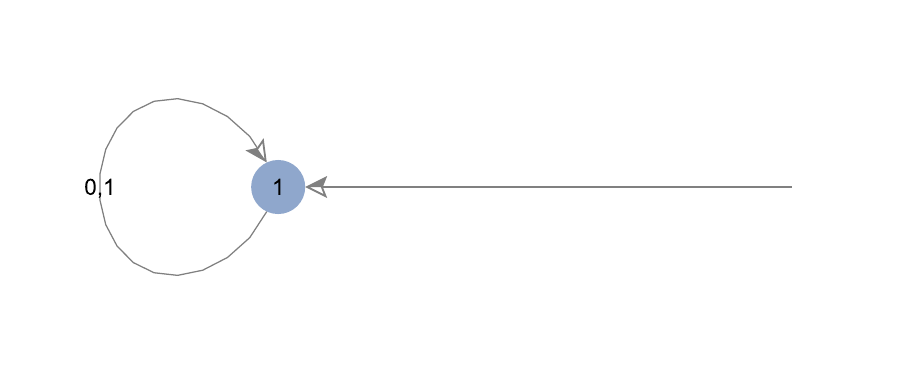}}
\end{center}

Now let $i = 1$.
Since there is a single element in the $2$-kernel of $T_n \bmod 2$, we can write $\iota(f_0(x))$ as
\[
	\iota(f_0(x)) = \iota(f_0(x^2)) + x \, \iota(f_0(x^2)).
\]
Let $j = 2$.
Define $\delta(x) \in \Z_2\llbracket x \rrbracket$ by writing $\iota(f_0(x)) = f(x) - 2 \delta(x)$.
Then
\[
	f(x) - 2 \delta(x) = f(x^2) - 2 \delta(x^2) + x \left(f(x^2) - 2 \delta(x^2)\right),
\]
so
\begin{align*}
	\frac{f(x) - f(x^2) - x f(x^2)}{2}
	&= \delta(x) - \delta(x^2) - x \delta(x^2) \\
	&\equiv \delta(x) - \delta(x)^2 - x \delta(x)^2 \mod 2.
\end{align*}
Use $P(x, y)$ to compute a polynomial $Q^{**}(x, y) \in \Z[x, y]$ such that
\[
	Q^{**}\left(x, \frac{f(x) - f(x^2) - x f(x^2)}{2}\right) = 0,
\]
keeping only one irreducible polynomial factor from each resultant.
The coefficients of $Q^{**}(x, y)$ are all divisible by $16$, so let
\[
	Q^*(x, y)
	\colonequal \pi\left(\frac{1}{16} Q^{**}(x, y - y^2 - x y^2)\right)
	= (x+1)^{16} y^8 + (x+1)^{10} y^2 + x^4
	\in \F_2[x, y].
\]
Then $Q^*(x, \pi(\delta(x))) = 0$.
Computing an Ore form for $Q^*(x, y)$ gives
\[
	(x+1)^{11} y^8 + x^2 (x+1)^3 y^4 + (x+1)^5 y^2 + x^2 y.
\]
Let $\Delta_0^{(2)}$ be a root of this polynomial that is congruent modulo $2$ to $\pi(\delta(x))$.
This concludes the loop over $j$.

Let $f_1(x) \colonequal \pi(\Delta_0^{(2)}(x))$.
We have computed a polynomial satisfied by $f_1(x) = \pi(\delta(x))$.
From this polynomial, one computes the following automaton $\mathcal M_1$ for the coefficients of $f_1(x)$, the $n$th term of which is the $2^1$ digit of $T_n$.
\begin{center}
	\scalebox{.8}{\includegraphics{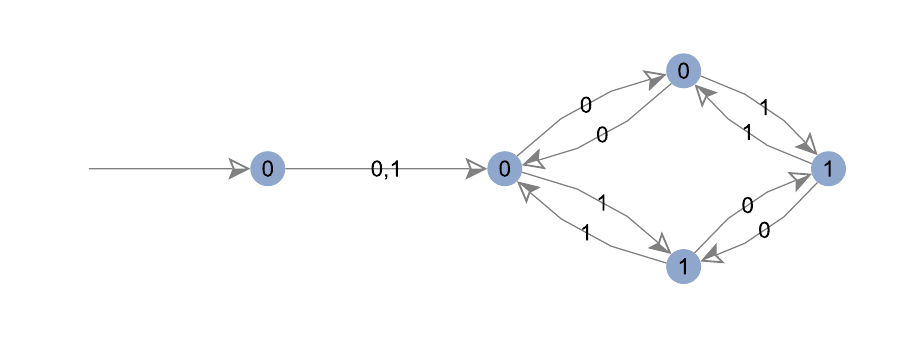}}
\end{center}
This concludes the loop over $i$.
Since $\mathcal M_0$ has only one state, the product $\mathcal M_0 \times \mathcal M_1$ is the following, simply a relabeling of the states of $\mathcal M_1$.
\begin{center}
	\scalebox{.8}{\includegraphics{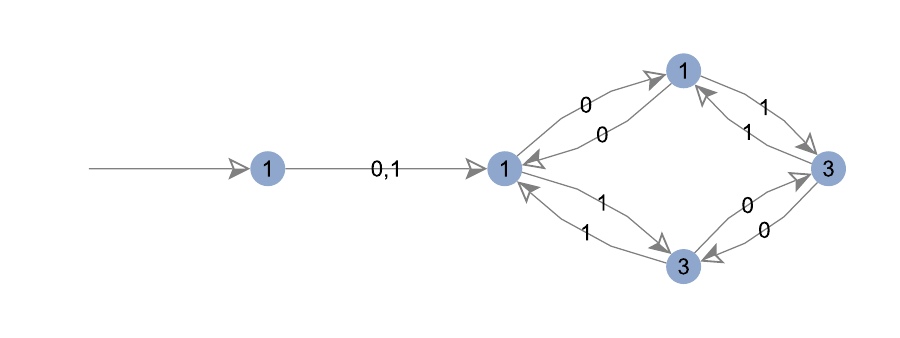}}
\end{center}

We have not been able to carry out the computation for $T_n \bmod 8$.
The next step (for $i = 1$ and $j = 3$) would be to compute $F_0^{(2)}(x) = f(x) - p \Delta_0^{(2)}(x)$.
This computation is not difficult, and $F_0^{(2)}$ is an algebraic series of degree $14$.
However, we have not been able to compute the sum $F_0^{(2)}(x) - (x + 1) F_0^{(2)}(x^2)$, which is required next.
Even if we had, however, there is still another Ore polynomial computation to undertake.

\section{Multidimensional diagonals}\label{Lucas}

\subsection{Lucas products}

Lucas' well-known theorem on binomial coefficients modulo a prime states that if $n = n(0) p^0 + n(1) p^1 + \dots + n(l) p^l$ and $m = m(0) p^0 + m(1) p^1 + \dots + m(l) p^l$ in base $p$, then
\[
	\binom{n}{m} \equiv \prod_{i=0}^l \binom{n(i)}{m(i)} \mod p.
\]
Lucas-type results are also known for the Ap\'ery numbers and other sequences, such as  the constant term of $P(x)^n$ for certain Laurent polynomials $P(x)$~\cite{Samol--vanStraten}.
Since such results hold for general $p$, they fall outside the scope of the previous sections.
In this section we combine ideas from Section~\ref{easy_algorithm} and Section~\ref{hard_algorithm} to show that Lucas products exist for a large number of sequences.

As stated, Theorem~\ref{diagonal closure} applies to the ``full'' diagonal of a rational expression, which is a univariate power series obtained by collapsing all variables into one variable.
However, it is not difficult to see that Theorem~\ref{diagonal closure} generalizes to any diagonal obtained by collapsing any subsets of variables.
That is, let $B = \{b_1, \dots, b_l\}$ be a set partition of $\{1, 2, \dots, k\}$, and define $\gamma(i) = j$ if $i \in b_j$.
Then let
\[
	\mathcal{D}_B \left(\sum_{n_1, \dots, n_k \geq 0} a_{n_1, \dots, n_k} x_1^{n_1} \cdots x_k^{n_k}\right)
	\colonequal \sum_{n_1, \dots, n_l \geq 0} a_{n_{\gamma(1)}, \dots, n_{\gamma(k)}} x_1^{n_1} \cdots x_l^{n_l}.
\]
The full diagonal $\mathcal{D}$ is the diagonal $\mathcal{D}_{\{\{1, 2, \dots, k\}\}}$ in which the set partition $B$ contains a single set.

The coefficients of $\mathcal{D}_B(f)$ form a multidimensional sequence.
As is the case for the full diagonal, this sequence is $p$-automatic.

\begin{theorem}\label{general diagonal closure}
Let $R(x_1,\ldots, x_k)$ and $Q(x_1, \ldots,x_k)$ be polynomials in $\Z_p[x_1, \ldots, x_k]$ such that $Q(0,\ldots,0) \nequiv 0 \mod p$.
Let $\alpha\geq 1$, and let $B$ be a set partition of $\{1, 2, \dots, k\}$.
Then the $|B|$-dimensional sequence of coefficients of
\[
	\mathcal{D}_B\left(\frac{R(x_1, \ldots, x_k)}{Q(x_1, \ldots,x_k)}\right) \bmod p^\alpha
\]
is $p$-automatic.
\end{theorem}

The proof of Theorem~\ref{general diagonal closure} is identical to the proof of Theorem~\ref{diagonal closure} except for the last step, where instead of restricting to $\mu_{d, \dots, d}$ one considers the more general operator $\mu_{d_{\gamma(1)}, \dots, d_{\gamma(k)}}$ for $d_1, \dots, d_l \in \{0, 1, \dots, p-1\}$.
Algorithms~\ref{slow diagonal algorithm} and \ref{fast diagonal algorithm} generalize accordingly.
We mention that Denef and Lipshitz~\cite[Theorem~5.2]{Denef--Lipshitz} prove a generalization to algebraic power series.

For example, the bivariate generating function for $\binom{n}{m}$ is rational, so Theorem~\ref{general diagonal closure} applies with the diagonal $\mathcal{D}_{\{\{1\}, \{2\}\}}$ that collapses no variables.
Therefore, for any fixed prime we can compute an automaton that outputs $\binom{n}{m} \bmod p$ when fed the base-$p$ digits of $n$ and $m$ as the sequence of pairs $(n(0), m(0)), \dots, (n(l), m(l))$.
This automaton corresponds to the Lucas product for binomial coefficients modulo $p$.

Of course, we would like to prove theorems for arbitrary $p$ when possible.
This can be done by putting the polynomial in Ore form.
A general result using this approach is as follows.

\begin{theorem}\label{Lucas product}
Let $s \geq 1$ be an integer, and let $p$ be a prime such that $p \equiv 1 \mod s$.
Let $Q = Q(x_1, \dots, x_k) \in \Z_p[x_1, \dots, x_k]$ be a polynomial such that $Q(0, \dots, 0) = 1$.
Let
\[
	f = f(x_1, \dots, x_k) = \frac{1}{Q^{1/s}} = \sum_{n_1, \dots, n_k \geq 0} a_{n_1, \dots, n_k} x_1^{n_1} \cdots x_k^{n_k} \in \Z_p\llbracket x_1, \dots, x_k \rrbracket.
\]
Let $B$ be a set partition of $\{1, 2, \dots, k\}$.
If $\Lambda_{d_{\gamma(1)}, \dots, d_{\gamma(k)}}(Q^{(p-1)/s}) \bmod p$ is a constant polynomial for each $(d_1, \dots, d_{|B|}) \in \{0, 1, \dots, p - 1\}^{|B|}$, then there is a Lucas product for the coefficients of $\mathcal{D}_B(f)$ modulo $p$.
Namely,
\[
	a_{n_{\gamma(1)}, \dots, n_{\gamma(k)}} \equiv \prod_{i=0}^l a_{n_{\gamma(1)}(i), \dots, n_{\gamma(k)}(i)} \mod p,
\]
where $l$ is the maximum length of the base-$p$ representations of $n_1, \dots, n_{|B|}$.
\end{theorem}

\begin{proof}
Write
\[
	Q^{(p-1)/s} = \sum_{m_1, \dots, m_k \geq 0} c_{m_1, \dots, m_k} x_1^{m_1} \cdots x_k^{m_k}.
\]
Let $0 \leq d_j \leq p-1$ for each $1 \leq j \leq |B|$.
Since $\Lambda_{d_{\gamma(1)}, \dots, d_{\gamma(k)}}(Q^{(p-1)/s}) \bmod p$ is a constant polynomial, it is congruent modulo $p$ to $c_{d_{\gamma(1)}, \dots, d_{\gamma(k)}}$.
We have $1 = Q \, f^s$.
Raising both sides to the power $(p-1)/s$ and multiplying by $f$ gives
\[
	f = Q^{(p-1)/s} f^p,
\]
which is in Ore form.
By Proposition~\ref{Cartier extraction},
\begin{align*}
	\Lambda_{d_{\gamma(1)}, \dots, d_{\gamma(k)}}(f)
	&= \Lambda_{d_{\gamma(1)}, \dots, d_{\gamma(k)}}(Q^{(p-1)/s} f^p) \\
	&\equiv \Lambda_{d_{\gamma(1)}, \dots, d_{\gamma(k)}}(Q^{(p-1)/s}) f \mod p \\
	&\equiv c_{d_{\gamma(1)}, \dots, d_{\gamma(k)}} f \mod p.
\end{align*}
Therefore, composing Cartier operators of this form results in a product of the corresponding coefficients.
It remains to show that
\[
	c_{d_{\gamma(1)}, \dots, d_{\gamma(k)}}
	\equiv a_{d_{\gamma(1)}, \dots, d_{\gamma(k)}} \mod p.
\]
To see this, write $Q^{1/s} = 1 + g$ for some $g \in \Z_p\llbracket x_1, \dots, x_k \rrbracket$ with $g(0, \dots, 0) = 0$.
Then $Q^{p/s} \equiv 1 + g^p \mod p$, so $Q^{(p-1)/s} \equiv Q^{-1/s} + g^p Q^{-1/s} \mod p$.
Each nonzero term of $g^p$ has total degree at least $p$, so the series $Q^{(p-1)/s}$ and $\frac{1}{Q^{1/s}}$ are congruent modulo $p$ on terms with total degree less than $p$.
\end{proof}

In the previous proof, each element in the kernel of $\mathcal{D}_B(f) \bmod p$ belongs to the set
\[
	\{\mathcal{D}_B(f) \bmod p, \, \mathcal{D}_B(2 f) \bmod p, \, \ldots, \, \mathcal{D}_B((p-1) f) \bmod p, \, 0\},
\]
so in particular there is an automaton for the coefficients modulo $p$ containing at most $p$ states.

Lucas' theorem is a simple corollary of Theorem~\ref{Lucas product}.
Recall that the generating function for the binomial coefficients is
\[
	f(x_1, x_2) = \sum_{n \geq 0} \sum_{m \geq 0} \binom{n}{m} x_1^n x_2^m = \frac{1}{1 - x_1 - x_1 x_2}.
\]
Let $B = \{\{1\}, \{2\}\}$, $Q = 1 - x_1 - x_1 x_2$, and $s = 1$.
For $0 \leq d \leq p-1$ and $0 \leq e \leq p-1$ the polynomial $\Lambda_{d,e} (Q^{p-1})$ is a constant since $\deg_{x_1} Q = \deg_{x_2} Q = 1$.
Therefore the theorem applies.
Alternatively, one can verify directly that
\begin{align*}
	\Lambda_{d,e} ((1-x_1-x_1x_2)^{p-1})
	&= \Lambda_{d,e} \left ( \sum_{k=0}^{p-1}{p-1 \choose k } (-x_1)^{k}(1+x_2)^{k}\right ) \\
	&= \Lambda_{d,e} \left ( \sum_{k=0}^{p-1}{p-1 \choose k } (-x_1)^{k}
   \sum_{l=0}^k {k \choose l }x_2^l \right ) \\
	&= (-1)^{d}{p-1 \choose d}{d \choose e} \\
	&\equiv {d\choose e} \mod p,
\end{align*}
since 
\begin{align*}
	(-1)^{d} {p-1 \choose d}
	&= (-1)^{d} \cdot \frac{(p-1)!}{(p-1-d)! \, d!} \\
	&= \frac{(p-1)!}{(p-1-d) ! (-d) \cdots (-2)(-1)} \\
	&\equiv \frac{(p-1)!}{(p-1-d) ! (p-d) \cdots (p-2)(p-1)} \mod p \\
	&= 1.
\end{align*}

Central trinomial coefficients modulo $p$ also have a Lucas product, proved by Deutsch and Sagan~\cite[Theorem~4.7]{Deutsch--Sagan}.
Recall from Section~\ref{Central trinomial coefficients} that the generating function satisfies
\[
	1 = -(x + 1) (3 x - 1) f(x)^2.
\]
For $p \neq 2$, the degree of $(-(x + 1) (3 x - 1))^{(p-1)/2}$ is $p - 1$, so the conditions of Theorem~\ref{Lucas product} are satisfied.
This also shows that a Lucas product holds for a general family of sequences considered by Noe~\cite{Noe}.

Gessel's Lucas product for the Ap\'ery numbers is also a corollary of Theorem~\ref{Lucas product}.
Let
\[
	Q = (1 - x_1 - x_2) (1 - x_3 - x_4) - x_1 x_2 x_3 x_4
\]
and $B = \{\{1, 2, 3, 4\}\}$.
For each $1 \leq i \leq 4$, we have $\deg_{x_i} Q = 1$, so $\deg_{x_i} Q^{p-1} = p-1$.
Therefore $\Lambda_{d, d, d, d}(Q^{p-1}) \bmod p$ is a constant polynomial, so Theorem~\ref{Lucas product} applies for every prime $p$.

\subsection{Binomial coefficients modulo a prime power}

We can take a similar approach to diagonal sequences modulo $p^\alpha$.
One could write out conditions under which a Lucas product exists, and this would account for Theorem~\ref{Apery mod 9} and even Theorem~\ref{Apery mod 25}.

Here we restrict our attention to binomial coefficients.
Generalizations of Lucas' theorem to prime powers have been given by Granville~\cite{Granville} and by Davis and Webb~\cite{Davis--Webb}.
In general the sequences in the $p$-kernel of $\left(\binom{n}{m} \bmod p^\alpha\right)_{n \geq 0, m \geq 0}$ are not multiples of the original sequence, precluding a Lucas product.
However, we may still carry out the computations symbolically to obtain a third generalization of Lucas' theorem to prime powers.
The result is the following.
Write $D = \{0, 1, \dots, p^\alpha - p^{\alpha - 1}\}$.

\begin{theorem}\label{prime power Lucas}
Let $p$ be a prime, and let $\alpha \geq 1$.
If $n = n_l \cdots n_1 n_0$ and $m = m_l \cdots m_1 m_0$ in base $p$, then
\[
	\binom{n}{m} \equiv
	\sum_{\substack{(i_0, \dots, i_l) \in D^{l+1} \\ (j_0, \dots, j_l) \in D^{l+1}}}
		(-1)^{n - i + \sum_{h=0}^l i_h}
		\binom{p^{\alpha - 1} - 1}{n-i}
		\binom{n-i}{m-j}
		\prod_{h=0}^l \binom{p^\alpha - p^{\alpha - 1}}{i_h} \binom{i_h}{j_h}
	\mod p^\alpha,
\]
where $i = \sum_{h=0}^l i_h p^h$ and $j = \sum_{h=0}^l j_h p^h$.
\end{theorem}

Note that $\sum_{h=0}^l i_h p^h$ and $\sum_{h=0}^l j_h p^h$ are representations of integers in base $p$ with an enlarged digit set $D$ rather than the standard digit set $\{0, \dots, p-1\}$.

\begin{proof}
For convenience, let $\phi(p^\alpha) \colonequal p^\alpha - p^{\alpha - 1}$.
(This notation is justified since this is the value of the Euler totient function at a prime power.)

Recall that the generating function for $\binom{n}{m}$ is
\[
	f = \frac{1}{1 - x_1 - x_1 x_2}.
\]
First we show that every element in the $p$-kernel of $f \bmod p^\alpha$ is of the form $g \cdot f^{p^{\alpha - 1}}$ for some polynomial $g$ with coefficients in $\Z/(p^\alpha \Z)$.
Raising both sides of $1 = (1 - x_1 - x_1 x_2) f$ to the power $p^{\alpha - 1} - 1$ and multiplying by $f$ gives
\[
	f = (1 - x_1 - x_1 x_2)^{p^{\alpha - 1} - 1} f^{p^{\alpha - 1}},
\]
so $f$ itself is of the form $g \cdot f^{p^{\alpha - 1}}$.
Moreover, the image of $g \cdot f^{p^{\alpha - 1}}$ under the Cartier operator is of the same form:
Raise both sides of $1 = (1 - x_1 - x_1 x_2) f$ to the power $p^\alpha - p^{\alpha - 1}$ and multiply by $f^{p^{\alpha - 1}}$ to obtain $f^{p^{\alpha - 1}} = (1 - x_1 - x_1 x_2)^{\phi(p^\alpha)} \cdot f^{p^\alpha}$; then
\begin{align*}
	\Lambda_{n_h, m_h} \left( g \cdot f^{p^{\alpha - 1}} \right)
	&= \Lambda_{n_h, m_h} \left( g \cdot (1 - x_1 - x_1 x_2)^{\phi(p^\alpha)} \cdot f^{p^\alpha} \right) \\
	&\equiv \Lambda_{n_h, m_h} \left( g \cdot (1 - x_1 - x_1 x_2)^{\phi(p^\alpha)} \right) \cdot f^{p^{\alpha - 1}} \mod p^\alpha
\end{align*}
by Proposition~\ref{Cartier extraction}.

Next we determine the coefficients of $\Lambda_{n_h, m_h}(g \cdot f^{p^{\alpha - 1}}) / f^{p^{\alpha - 1}}$ in terms of the coefficients of $g$.
Write $g = \sum_{k,l} c_{k,l} x_1^k x_2^l$, where the sum is over all $k \in \Z$ and $l \in \Z$, so that $c_{k,l} = 0$ if $k$ or $l$ is negative.
Expanding $(1 - x_1 - x_1 x_2)^{\phi(p^\alpha)}$ gives
\begin{align*}
	\Lambda_{n_h, m_h} \left( g \cdot f^{p^{\alpha - 1}} \right)
	&\equiv \Lambda_{n_h, m_h} \left( \sum_{k,l} \sum_{i=0}^{\phi(p^\alpha)} \sum_{j=0}^{\phi(p^\alpha)} (-1)^i \binom{\phi(p^\alpha)}{i} \binom{i}{j} c_{k,l} x_1^{k+i} x_2^{l+j} \right) \cdot f^{p^{\alpha - 1}} \mod p^\alpha \\
	&= \Lambda_{n_h, m_h} \left( \sum_{k,l} \sum_{i=0}^{\phi(p^\alpha)} \sum_{j=0}^{\phi(p^\alpha)} (-1)^i \binom{\phi(p^\alpha)}{i} \binom{i}{j} c_{k-i,l-j} x_1^k x_2^l \right) \cdot f^{p^{\alpha - 1}} \\
	&= \left( \sum_{k,l} \sum_{i=0}^{\phi(p^\alpha)} \sum_{j=0}^{\phi(p^\alpha)} (-1)^i \binom{\phi(p^\alpha)}{i} \binom{i}{j} c_{p k + n_h - i,p l + m_h - j} x_1^k x_2^l \right) \cdot f^{p^{\alpha - 1}}.
\end{align*}
Therefore the coefficient of $x_1^k x_2^l$ in $\Lambda_{n_h, m_h}(g \cdot f^{p^{\alpha - 1}}) / f^{p^{\alpha - 1}}$ is congruent modulo $p^\alpha$ to
\[
	\sum_{i \in D} \sum_{j \in D} (-1)^i \binom{\phi(p^\alpha)}{i} \binom{i}{j} c_{p k + n_h - i,p l + m_h - j}.
\]

The binomial coefficient $\binom{n}{m}$ is simply the constant term of $\Lambda_{n_l, m_l} \cdots \Lambda_{n_1, m_1} \Lambda_{n_0, m_0}(f)$.
By iterating the expression we have just computed, we see that $\binom{n}{m}$ is congruent modulo $p^\alpha$ to
\[
	\binom{n}{m} \equiv
	\sum_{\substack{(i_0, \dots, i_l) \in D^{l+1} \\ (j_0, \dots, j_l) \in D^{l+1}}}
		(-1)^{\sum_{h=0}^l i_h}
		c_{\sum_{h=0}^l (n_h - i_h) p^h, \sum_{h=0}^l (m_h - j_h) p^h}
		\prod_{h=0}^l \binom{\phi(p^\alpha)}{i_h} \binom{i_h}{j_h}
	\mod p^\alpha,
\]
where $c_{k,l}$ is defined by
\[
	\sum_{k,l} c_{k,l} x_1^k x_2^l
	\colonequal (1 - x_1 - x_1 x_2)^{p^{\alpha - 1} - 1}
\]
so that $\left( \sum_{k,l} c_{k,l} x_1^k x_2^l \right) \cdot f^{p^{\alpha - 1}} = f$.
Therefore it suffices to compute $c_{k,l}$.
We have
\begin{align*}
	\sum_{k,l} c_{k,l} x_1^k x_2^l
	&= (1 - x_1 - x_1 x_2)^{p^{\alpha - 1} - 1} \\
	&= \sum_{k=0}^{p^{\alpha - 1} - 1} \sum_{l=0}^{p^{\alpha - 1} - 1} (-1)^k \binom{p^{\alpha - 1} - 1}{k} \binom{k}{l} x_1^k x_2^l.
\end{align*}
Therefore $c_{k,l} = (-1)^k \binom{p^{\alpha - 1} - 1}{k} \binom{k}{l}$, and this gives the expression claimed.
\end{proof}

The sum in Theorem~\ref{prime power Lucas} is quite large.
It is possible to restrict the indices considerably (since for example the summand is $0$ if $j_h > i_h$ for any $h$), but we have not been able to extract an algorithm for computing $\binom{n}{m} \bmod p^\alpha$ that rivals Granville's algorithm in speed.
However, the expression in Theorem~\ref{prime power Lucas} gives some indication of what analogous expressions for some other two-dimensional rational sequences look like.

\section*{Acknowledgments}

The authors would like to thank Mathieu Guay-Paquet, Marcus Pivato and David Poole for helpful discussions.
The authors would also like to thank the referee for a careful reading of an earlier version of the paper. Finally, we thank the Seasoned Spoon for being the catalyst for the crucial insight.

{\footnotesize
\bibliographystyle{alpha}
\bibliography{bibliography}

\def\ocirc#1{\ifmmode\setbox0=\hbox{$#1$}\dimen0=\ht0 \advance\dimen0
  by1pt\rlap{\hbox to\wd0{\hss\raise\dimen0
  \hbox{\hskip.2em$\scriptscriptstyle\circ$}\hss}}#1\else {\accent"17 #1}\fi}
\begin{thebibliography}{CKMFR80}

\bibitem[AB12]{Adamczewski--Bell--2}
Boris Adamczewski and Jason~P. Bell.
\newblock On vanishing coefficients of algebraic power series over fields of
  positive characteristic.
\newblock {\em Inventiones Mathematicae}, 187(2):343--393, 2012.

\bibitem[AB13]{Adamczewski--Bell}
Boris Adamczewski and Jason~P. Bell.
\newblock Diagonalization and rationalization of algebraic {L}aurent series.
\newblock {\em Annales Scientifiques de l'\'Ecole Normale Sup\'erieure},
  46(6):963--1004, 2013.

\bibitem[AK73]{Alter--Kubota}
Ronald Alter and K.~K. Kubota.
\newblock Prime and prime power divisibility of {C}atalan numbers.
\newblock {\em Journal of Combinatorial Theory, Series A}, 15(3):243--256,
  1973.

\bibitem[ARS09]{Allouche--Rampersad--Shallit}
Jean-Paul Allouche, Narad Rampersad, and Jeffrey Shallit.
\newblock Periodicity, repetitions, and orbits of an automatic sequence.
\newblock {\em Theoretical Computer Science}, 410(30--32):2795--2803, 2009.

\bibitem[AS92]{Allouche--Shallit}
Jean-Paul Allouche and Jeffrey Shallit.
\newblock The ring of $k$-regular sequences.
\newblock {\em Theoretical Computer Science}, 98(2):163--197, 1992.

\bibitem[AS03]{ash}
Jean-Paul Allouche and Jeffrey Shallit.
\newblock {\em Automatic Sequences: Theory, Applications, Generalizations}.
\newblock Cambridge University Press, Cambridge, 2003.

\bibitem[Atk98]{Atkinson}
M.~D. Atkinson.
\newblock Permutations which are the union of an increasing and a decreasing
  subsequence.
\newblock {\em The Electronic Journal of Combinatorics}, 5:{\#}R6, 1998.

\bibitem[Beu95]{Beukers}
Frits Beukers.
\newblock Consequences of {A}p\'ery's work on $\zeta(3)$.
\newblock {\em Rencontres Arithm\'etiques de Caen}, 1995.

\bibitem[B{\'o}n98]{Bona}
Mikl\'{o}s B{\'o}na.
\newblock The permutation classes equinumerous to the smooth class.
\newblock {\em The Electronic Journal of Combinatorics}, 5:{\#}R31, 1998.

\bibitem[CCC80]{Chowla--Cowles--Cowles}
S.~Chowla, J.~Cowles, and M.~Cowles.
\newblock Congruence properties of {A}p\'ery numbers.
\newblock {\em Journal of Number Theory}, 12(2):188--190, 1980.

\bibitem[Chr74]{Christol}
Gilles Christol.
\newblock El\'ements analytiques uniformes et multiformes.
\newblock {\em S\'eminaire Delange-Pisot-Poitou. Th\'eorie des nombres},
  1(15):1--18, 1973--1974.

\bibitem[CKMFR80]{CKMR}
Gilles Christol, Teturo Kamae, Michel Mend{\`e}s~France, and G\'erard Rauzy.
\newblock Suites alg\'ebriques, automates et substitutions.
\newblock {\em Bulletin de la Soci\'et\'e Math\'ematique de France},
  108(4):401--419, 1980.

\bibitem[Del13]{Delaygue}
Eric Delaygue.
\newblock Arithmetic properties of {A}p\'ery-like numbers.
\newblock 2013.
\newblock \url{http://arxiv.org/abs/1310.4131}.

\bibitem[DL87]{Denef--Lipshitz}
Jan Denef and Leonard Lipshitz.
\newblock Algebraic power series and diagonals.
\newblock {\em Journal of Number Theory}, 26(1):46--67, 1987.

\bibitem[DS06]{Deutsch--Sagan}
Emeric Deutsch and Bruce~E. Sagan.
\newblock Congruences for {C}atalan and {M}otzkin numbers and related
  sequences.
\newblock {\em Journal of Number Theory}, 117(1):191--215, 2006.

\bibitem[DW90]{Davis--Webb}
Kenneth~S. Davis and William~A. Webb.
\newblock Lucas' theorem for prime powers.
\newblock {\em European Journal of Combinatorics}, 11(3):229--233, 1990.

\bibitem[Eil74]{ei}
Samuel Eilenberg.
\newblock {\em Automata, languages, and machines. {V}ol. {A}}.
\newblock Academic Press [A subsidiary of Harcourt Brace Jovanovich,
  Publishers], New York, 1974.
\newblock Pure and Applied Mathematics, Vol. 58.

\bibitem[ELY08]{Eu--Liu--Yeh}
Sen-Peng Eu, Shu-Chung Liu, and Yeong-Nan Yeh.
\newblock Catalan and {M}otzkin numbers modulo 4 and 8.
\newblock {\em European Journal of Combinatorics}, 29(6):1449--1466, 2008.

\bibitem[Fur67]{Fur}
Harry Furstenberg.
\newblock Algebraic functions over finite fields.
\newblock {\em Journal of Algebra}, 7:271--277, 1967.

\bibitem[Ges82]{Gessel}
Ira Gessel.
\newblock Some congruences for {A}p\'ery numbers.
\newblock {\em Journal of Number Theory}, 14(3):362--368, 1982.

\bibitem[Gra97]{Granville}
Andrew Granville.
\newblock Arithmetic properties of binomial coefficients. {I}. {B}inomial
  coefficients modulo prime powers.
\newblock In {\em Organic Mathematics ({B}urnaby, {BC}, 1995)}, volume~20 of
  {\em CMS Conf. Proc.}, pages 253--276. Amer. Math. Soc., Providence, RI,
  1997.

\bibitem[KKM12]{Kauers--Krattenthaler--Muller}
Manuel Kauers, Christian Krattenthaler, and Thomas~W. M\"{u}ller.
\newblock A method for determining the mod-$2^k$ behaviour of recursive
  sequences, with applications to subgroup counting.
\newblock {\em The Electronic Journal of Combinatorics}, 18(2):{\#}P37, 2012.

\bibitem[KM13]{Krattenthaler--Muller}
Christian Krattenthaler and Thomas~W. M\"{u}ller.
\newblock A method for determining the mod-3$^k$ behaviour of recursive
  sequences.
\newblock 2013.
\newblock \url{http://arxiv.org/abs/1308.2856}.

\bibitem[KM14]{Krattenthaler--Muller2014}
Christian Krattenthaler and Thomas~W. M\"{u}ller.
\newblock Generalised {A}p\'ery numbers modulo 9.
\newblock 2014.
\newblock \url{http://arxiv.org/abs/1401.1444}.

\bibitem[Le05]{Le}
Ian Le.
\newblock Wilf classes of pairs of permutations of length $4$.
\newblock {\em The Electronic Journal of Combinatorics}, 12:{\#}R25, 2005.

\bibitem[Lin12]{Lin}
Hsueh-Yung Lin.
\newblock Odd {C}atalan numbers modulo {$2^k$}.
\newblock {\em Integers}, 12(2):161--165, 2012.

\bibitem[LY10]{Liu--Yeh}
Shu-Chung Liu and Jean C.-C. Yeh.
\newblock Catalan numbers modulo $2^k$.
\newblock {\em Journal of Integer Sequences}, 13:10.5.4, 2010.

\bibitem[Noe06]{Noe}
Tony~D. Noe.
\newblock On the divisibility of generalized central trinomial coefficients.
\newblock {\em Journal of Integer Sequences}, 9(2):Article 06.2.7, 12, 2006.

\bibitem[Pet03]{Peter}
Manfred Peter.
\newblock The asymptotic distribution of elements in automatic sequences.
\newblock {\em Theoretical Computer Science}, 301(1-3):285--312, 2003.

\bibitem[Row10]{Rowland_binary_trees}
Eric Rowland.
\newblock Pattern avoidance in binary trees.
\newblock {\em Journal of Combinatorial Theory, Series A}, 117(6):741--758,
  2010.

\bibitem[RY12]{Rowland--Yassawi}
Eric Rowland and Reem Yassawi.
\newblock A characterization of $p$-automatic sequences as columns of linear
  cellular automata.
\newblock 2012.
\newblock \url{http://arxiv.org/abs/1209.6008}.

\bibitem[RZ14]{Rowland--Zeilberger}
Eric Rowland and Doron Zeilberger.
\newblock A case study in meta-automation: automatic generation of congruence
  automata for combinatorial sequences.
\newblock {\em Journal of Difference Equations and Applications}, to appear,
  2014.

\bibitem[Sal87]{Salon}
Olivier Salon.
\newblock Suites automatiques \`a multi-indices et alg\'ebricit\'e.
\newblock {\em Comptes Rendus des S\'eances de l'Acad\'emie des Sciences.
  S\'erie I. Math\'ematique}, 305(12):501--504, 1987.

\bibitem[Slo]{OEIS}
Neil Sloane.
\newblock The {O}n-{L}ine {E}ncyclopedia of {I}nteger {S}equences.
\newblock http://oeis.org.

\bibitem[Str14]{Straub}
Armin Straub.
\newblock Multivariate {A}p\'ery numbers and supercongruences of rational
  functions.
\newblock 2014.
\newblock \url{http://arxiv.org/abs/1401.0854}.

\bibitem[SvS09]{Samol--vanStraten}
Kira Samol and Duco van Straten.
\newblock Dwork congruences and reflexive polytopes.
\newblock 2009.
\newblock \url{http://arxiv.org/abs/0911.0797}.

\bibitem[XX11]{Xin--Xu}
Guoce Xin and Jing-Feng Xu.
\newblock A short approach to {C}atalan numbers modulo $2^r$.
\newblock {\em The Electronic Journal of Combinatorics}, 18:{\#}P177, 2011.

\end{thebibliography}
}

\end{document}